\documentclass[10pt]{article}
\usepackage{amsfonts}
\usepackage{}
\usepackage[notcite,notref]{} %para que los labels sean visibles
\usepackage{blindtext}
\usepackage[a4paper, total={6.6in, 8.7in}]{geometry}
\usepackage{amsthm}
\usepackage{CJK}
\usepackage{graphicx}
\usepackage{amsmath}
\usepackage{setspace}
\usepackage{xcolor}
\usepackage{mathtools}
\usepackage{blkarray}

\theoremstyle{definition}
\newtheorem{theorem}{Theorem}
\newtheorem{lemma}{Lemma}

\newtheorem{definition}{Definition}

\newtheorem{example}{Example}

\newcounter{nex}
\usepackage[bitstream-charter]{mathdesign}
\usepackage[T1]{fontenc}
\usepackage{bbm}

\newcommand{\CC}{\mathbbmss{C}}

\DeclareMathOperator{\ind}{ind}

\DeclareMathOperator{\matt}{\bf mat}
\DeclareMathOperator{\diag}{diag}

\DeclareMathOperator{\squeeze}{squeeze}

\usepackage[english, algosection, algoruled, noline]{algorithm2e}

\usepackage{eucal}

\newcommand{\gc}{\mathcal{G}}
\newcommand{\hc}{\mathcal{H}}

\newcommand{\cc}{\mathcal{C}}
\newcommand{\dc}{\mathcal{D}}
\newcommand{\ic}{\mathcal{I}}
\newcommand{\pc}{\mathcal{P}}

\newcommand{\xc}{\mathcal{X}}

\newcommand{\zc}{\mathcal{Z}}

\newcommand{\tc}{\mathcal{T}}

\begin{document}
%\begin{spacing}{1.2}

\title{On the extensions of the GD inverse of tensors via the M-Product}

\author{
Hongwei Jin,\thanks{School of Mathematics and physics, Guangxi Minzu University, 530006, Nanning, PR China; Center for Applied Mathematics of Guangxi, Guangxi Minzu University, Nanning, 530006, PR China. { E-mail address}:  jhw$\_$math@126.com} \quad
Siran Chen,\thanks{School of Mathematics and physics, Guangxi Minzu University, 530006,
Nanning, PR China. { E-mail address}:  18877545298@163.com}\quad
Shaowu Huang,\thanks{Corresponding author. School of Mathematics and Finance, Putian University, Putian, PR China. { E-mail address}:  shaowu2050@126.com}\quad
Predrag S. Stanimirovi\'{c}\thanks{Faculty of Sciences and Mathematics, University of Ni\u{s}, Ni\u{s}, Serbia.  E-mail address:   pecko@pmf.ni.ac.rs}
}

\date{}

\maketitle

\begin{abstract}
 We study extensions of the GD tensor inverse using the  M-product.
 The aim of current research is threefold.
 In the first place, the tensor GD inverse under the M-product is introduced and considered.
 We give the several properties and representations of the GD inverse using the core nilpotent decomposition and then establish the reverse-order law rules for the GD inverse.
 Second, the tensor GDMP inverse is studied and the corresponding numerical algorithm is given.
In addition, the reverse- and forward-order laws of the GDMP inverse are established.
 Third, the GD-Star tensor inverse under the M-product is introduced and studied.
 Finally, the GD inverse, GDMP inverse and GD-Star inverse solutions of multilinear equations are investigated.
Illustrative numerical calculation is performed.


\end{abstract}

% % % % % % % % % % % % % % % % % % % % %

\section{Introduction}

Tensor quantization is typically an effective way to encode and analyze multidimensional data, such as images, movies, sounds, signals, etc.
Thus, the study of multidimensional arrays, or tensors, has gained popularity in recent years.
An array of multidimensional data can be considered as a tensor of the form $\gc=(\gc_{{i_1}i_2\ldots{i_m}})\in\mathbb{C}^{{n_1}\times{n_2}\times\ldots{n_m}}$ with entries over a complex field.
The positive integer $m$ is the order of the tensor $\gc$.
Scalars, vectors, and matrices are also known as zero-order, first-order, and second-order tensors, respectively.

The development of the tensor computation and multidimensional data analysis theories has drawn increasing attention over time.
Kilmer {\it et al.} {\rm \cite{K1}} first proposed the T-product.
Jin {\it et al.} {\rm \cite{J1}} proposed the cosine change product.
Kernfeld {\it et al.} {\rm \cite{KK}} proposed the M-product, and Einstein {\rm \cite{A1}} proposed the Einstein product.
Shao {\rm \cite{Shao1}} proposed the general product.
The $n$-mode product has been proposed by Qi {\rm \cite{L1}}.
Additionally, these tensor products have been studied in many fields of mathematics and applications, and they have garnered a lot of interest in recent works like: face recognition {\rm \cite{KBHH,Chen}}, low-rank tensor recovery {\rm \cite{Che,QW}}, robust tensor Principal Component Analysis {\rm \cite{kong1,L3}}, date completion and denoising {\rm \cite{Hu,Long,A2}}, image processing {\rm \cite{KBHH,MSL,Soltani,Tarzanagh}}, signal processing {\rm \cite{1,RU,2,3}}, computer vision {\rm \cite{4,5}}.

The predetermined time convergence neural networks for solving time-varying nonsingular multilinear tensor equations was studied in \cite{Wang4,Wang2}.
The  non-singular tensor equations by randomizing the Kaczmarz class method were also solved.
Wei {\it et al.} \cite{Wei} studied the neural network model for time-varying tensor complementarity problem.
Kernfeld {\it et al.} in \cite{KK} studied tensor products under invertible linear transformations.
Kilmer {\it et al.} in \cite{K3} investigated tensor-tensor products with application to optimal representation and compression.
K gave The decomposition of higher-order tensors and their applications was given by olda {\it et al.} in \cite{KB}.
Kilmer {\it et al.} in \cite{KM} explored an alternative representation of tensors associated with the matrix decomposition, proposed a new decomposition of tensors as the product of tensors, and defined the closed multiplication operation between tensors.
The tensor neural network model for generating the tensor singular value decomposition was studied in \cite{Wang}.
Wang {\it et al.} investigated the stochastic conditions of tensor functions based on tensor-tensor products \cite{Miao3}.
Further, Wang {\it et al.} in \cite{Wang1,Wang3} studied the tensor complementarity problem.
Shao {\it et al.} studied the asymmetric algebraic Riccati equations under tensor products \cite{Shao}.

The combination of tensors and generalized inverses can provide new methods and ideas for processing the high-dimensional data, and help to solve the complex mathematical and engineering problems in various fields.
Sun {\it et al.} used the Einstein product to define the Moore-Penrose (MP) inverse of tensors, and obtained explicit representations for the MP inverse of some tensor blocks \cite{Sun}.
Generalized tensor functions based on tensor singular value decomposition of the T-product are studied in \cite{Miao,Miao2}.
T-Jordan canonical form and the T-Drazin inverse were also studied.
Ji {\it et al.} in \cite{Ji2} defined the index of the  even order tensor and  generalized the concept of the Drazin inverse of the square matrices to the even order tensors.
Jin {\it et al.} \cite{Jin} used the tensor equations to establish tensor generalized inverses, studied the least squares solution of the tensor equation, and constructed an algorithm for computing the MP inverse of any tensor.
Liang {\it et al.} \cite{Liang} generalized the concept of  the MP inverse of matrices to further results of the Einstein product in the case of tensors and got its application to tensor approximation problems.
The authors of \cite{Che1} studied the  calculation of a complete orthogonal decomposition of a third-order tensor, known as the t-URV decomposition.
Cong {\it et al.} in \cite{Cong} showed that the Moore-inverse of the tensor can be represented by T-SVD with T-product, and gave the equivalent conditions for the stable perturbation of the MP inverse.
Further results on the generalized inverse of tensors on the Einstein product were investigated in \cite{Behera}.
The binuclear generalized inversions of the third-order dual tensors based on the T-product was studied in \cite{Liu}.
Cong {\it et al.} studied the characterization and perturbation of the EP inverse of the tensor kernel based on the T-product \cite{Cong2}.
The core and core-EP inverses of tensors was studied in \cite{Sahoo}.
Cao {\it et al.} studied the perturbation analysis of the tensor inversions and  tensor systems in terms of the T-product \cite{Cao,Cao1}.
Several tensor inequalities based on the T-product were also considered.
The weighted generalized tensor functions based on tensor products and their applications were studied in \cite{Liu1}.

The M-product of tensors is one of the important operations for  processing the data and information in the  multidimensional space.
The authors of \cite{SPBS} introduced the concepts of Drazin inverse and core-EP inverse of tensors over the M-product and studied the CMP, DMP, and MPD inverses of tensors.
Jin {\it et al.} \cite{JXWL} studied the MP inverse of tensors with M-product.
They also gave the conditions for understanding it as the least square solution or minimum norm solution, and discussed the necessary and sufficient conditions for the reverse order law of the  MP inverse.
The notions of the Drazin and core-EP inverses on tensors via M-product were originated and cosidered in \cite{[386]}.

In this context, we will study the extensions of the tensor GD inverse by the M-product.
This article is structured as follows. In the second section, we firstly give the terms and symbols needed for this work, and then introduce the M-product between two tensors. In the third section, we introduce the GD inverse under the M-product and give the equivalent expressions of the GD inverse by using the core nilpotent decomposition.
Then, we describe the reverse order law of the GD inverse.
Meanwhile, we study the GDMP inverse of the tensor and give the corresponding numerical algorithm for calculating this inverse.
We also give the reverse-order law and forward-order law of this inverse.
Finally, the GD-Star inverse of the tensor under the M-product is described.
The numerical calculation of the GD-Star inverse is established and the corresponding properties are given.
In the fourth section, as an application, solutions of the GD inverse, the GDMP inverse and the GD-Star inverse of multilinear equations are given.
Numerical examples are presented to verify obtained theoretical outcomes.

\section{Basic notions and preliminaries}\label{two}

Following standard notation, matrices and higher order tensors will be denoted like $A, \gc$, respectively.
Moreover, $\gc(i,:,:)$, $\gc(:,i,:)$ and $\gc(:, :, i)$ correspond to $i$th horizontal slice, lateral slice and frontal slice, respectively.
For simplicity, the $i$th frontal slice will be denoted by $\gc^{(i)}$.

%When fixing two indices of the third-order tensor, we can get the fiber. The mode-3
%fiber is also called tube, denoted as $\gc(i, j, :)$. We denote $\overline{\mathbf{a}}$ the tube of the tensor $\gc$.  We can vectorize a tube by $\mathbf{a}=vec(\overline{\mathbf{a}})$.

%\subsection{M-product of the third-order tensors}

In the following, we will explain the transformation $L(\gc)$ with domain over tensors.
The notation $i\in[n]$ will be used to denote the index range $i=1,\ldots,n$.

\begin{definition}{\rm \cite{KB}}
The k-mode product of a tensor $\gc \in \mathbb{C}^{\eta_1\times \eta_2\times\cdots\times \eta_p}$ with a matrix $N\in\mathbb{C}^{J\times \eta_k}$ is termed as $\gc\times_kN$
and expresses as
\begin{equation*}
  (\gc\times_kN)_{i_1 i_2\ldots i_{k-1}ji_{k+1}\ldots i_p},\ j\in[J].
\end{equation*}
\end{definition}
Notice that
\begin{equation}\label{bb}
  L(\gc)=\gc\times_{\!\!3}{M}, \ \ \ L^{-1}(\gc)=\gc\times_{\!\!3}{M}^{-1},
\end{equation}
such that ${M}\in\mathbb{C}^{\eta_3\times \eta_3}$ is a nonsingular matrix.

\begin{definition}{\rm \cite{KK}}
The face-wise product of $\gc\triangle\hc$ of $\gc \in \mathbb{C}^{\eta_1\times \eta_2\times \eta_3}$ and $\hc \in \mathbb{C}^{\eta_2 \times l\times \eta_3}$ is defined as
$$(\gc\triangle\hc)^{(i)}=\gc^{(i)}\hc^{(i)}, \ i\in[\eta_3],$$
wherein $(\gc\triangle\hc)^{(i)}$, $\gc^{(i)}$, $\hc^{(i)}$ denote the $i$th frontal slice of $\gc\triangle\hc$, $\gc$ and $\hc$, respectively.
\end{definition}

%\begin{definition}{\rm \cite{KKA}}
%Let $L : \CC^{1\times \eta_2\times \eta_3}\rightarrow\CC^{1\times \eta_2\times \eta_3}$ is an invertible linear transform. Define $vec(L(\overline{\mathbf{a}}))=M{\mathbf{a}}$, where  $\overline{\mathbf{a}}$
%is a tube of tensor $\gc$, $M$ is the corresponding matrix of transform $L$.
%\end{definition}

%${\star \above 0pt M}$
%$_{M}^{\star}$

\begin{definition}{\rm \cite{KK,K3}}
The $\,\star_{\!\!\! M} \, $ product between $\gc \in \mathbb{C}^{\eta_1\times \eta_2\times \eta_3}$ and $\hc \in \mathbb{C}^{\eta_2 \times l\times \eta_3}$ is defined by
\begin{equation}\label{XZ}
\gc\,\star_{\!\!\! M} \, \hc=L^{-1}[L(\gc)\triangle L(\hc)].
\end{equation}
\end{definition}
The product defined in \eqref{XZ} is known as the "M-product" between third-order tensors.
%In the following, we will explain how to compute $\gc\times_3N$.
In fact, $\gc\times_{\!\!3}N$ is computed by using the matrix-matrix product
\begin{equation}
  \mathcal{Y}=\gc\times_{\!\!3}N\Leftrightarrow\mathcal{Y}_{(3)}
  =N\gc_{(3)},
\end{equation}
where $\gc_{(3)}\in\CC^{\eta_3\times \eta_1\eta_2}$ denotes the mode-3 unfolding of $\gc$, which is got by the \textbf{squeeze} function proposed in \cite{KBHH}.
More precisely,
\begin{equation}\label{as}
 \gc_{(3)}=\big[\big(\squeeze(\overrightarrow{\gc}_1)\big)^T, \big(\squeeze(\overrightarrow{\gc}_2)\big)^T,\ldots, \big(\squeeze(\overrightarrow{\gc}_{\eta_2})\big)^T\big],
\end{equation}
where $\overrightarrow{\gc}_i$, $i\in\{1,\ldots,\eta_2\}$, are the lateral slices of $\gc$ and $\squeeze(\cdot)$ is defined by
\begin{equation*}
  A_j=\squeeze(\overrightarrow{\gc}_j)\Rightarrow A_j(i,k)=\overrightarrow{\gc}_j(i,1,k), \ i\in[\eta_1], \ j\in[\eta_2], \   k\in[\eta_3].
\end{equation*}

Applying (\ref{bb}),  it may be desirable to have  an alternative expression for (\ref{XZ}).
More precisely,
\begin{equation}\label{nb}
\gc\,\star_{\!\!\! M} \, \hc=L^{-1}[L(\gc)\triangle L(\hc)]=
[(\gc\times_3{M})\triangle(\hc\times_3{M})]\times_3{M}^{-1}.
\end{equation}
Notice that one can get the T-Product choosing ${M}$ as the normalized DFT matrix in (\ref{nb})(see  \cite{KM}).
An alternative definition of the M-product is given by
\begin{equation}\label{vvb}
\matt(\gc)=\diag\left(L(\gc)^{(1)},\ldots,L(\gc)^{(\eta_3)}\right)
=\diag\left((\widehat{\gc})^{(1)},\ldots,(\widehat{\gc})^{(\eta_3)}\right),
\end{equation}
and $\matt^{-1}(\cdot)$ is the reverse operation of $\matt(\cdot)$, that is
                       \begin{equation}\label{mmm}
                         \gc=\matt^{-1}[\matt(\gc)].
\end{equation}
 Then, one can has the alternative definition of the M-product.
\begin{definition}{\rm \cite{SPBS}}
The $\,\star_{\!\!\! M} \, $ product over arguments $\gc \in \mathbb{C}^{\eta_1\times \eta_2\times \eta_3}$ and $\hc \in \mathbb{C}^{\eta_2 \times l\times \eta_3}$ is defined by
\begin{equation}\label{zzz}
\gc\,\star_{\!\!\! M} \, \hc=\matt^{-1}[\matt(\gc)\cdot\matt(\hc)].
\end{equation}
\end{definition}

In the following, we investigate elementary operations based on the M-Product.

\begin{lemma}{\rm \cite{KK,K3}} % \label{ppa}
If $\gc, \hc, \cc$ are third-order tensors with proper size, then the subsequent statements are valid:
\begin{description}
  \item  [(a)] $\gc\,\star_{\!\!\! M} \, (\hc+\cc) = \gc\,\star_{\!\!\! M} \, \hc + \gc\,\star_{\!\!\! M} \, \cc$.
\item [(b)] $(\gc+\hc)\,\star_{\!\!\! M} \, \cc = \gc\,\star_{\!\!\! M} \, \cc + \hc\,\star_{\!\!\! M} \, \cc$.
\item [(c)] $(\gc\,\star_{\!\!\! M} \, \hc)\,\star_{\!\!\! M} \, \cc = \gc\,\star_{\!\!\! M} \, (\hc\,\star_{\!\!\! M} \, \cc)$.
\end{description}
\end{lemma}

\begin{definition} {\rm \cite{KK}}
The  {\em identity tensor} $\ic \in \mathbb{C}^{n \times n\times \eta_3}$ is determined by $L(\ic)^{(i)}={I}_n$, $i\in [\eta_3]$.
\end{definition}

According to the definition of $L(\ic)^{(i)}$, it can be concluded
$$L(\ic)\triangle L(\gc)=L(\gc)\triangle L(\ic)=L(\gc),$$
which implies
$$\gc\,\star_{\!\!\! M} \, \ic=\ic\,\star_{\!\!\! M} \, \gc=\gc.$$

\begin{definition}{\rm \cite{KK}}
The tensor $\gc \in \mathbb{C}^{\eta_1 \times \eta_1\times \eta_3}$ is invertible if there exists $\xc \in \mathbb{C}^{\eta_1 \times \eta_1\times \eta_3}$ such that
$$\gc\,\star_{\!\!\! M} \, \xc=\xc\,\star_{\!\!\! M} \, \gc=\ic.$$
In such conditions, $\xc$ represents the inverse $\gc^{-1}$ of $\gc$.
\end{definition}
\noindent In the following, we give the definition of the conjugate transpose of the tensor $\gc$.

\begin{definition}{\rm \cite{KK}}
The {\em conjugate-transpose} of $\gc \in \mathbb{C}^{\eta_1 \times \eta_2 \times \eta_3}$, marked with $\gc^*$, is the $\eta_2 \times \eta_1 \times \eta_3$ tensor defined by $L(\gc^*)^{(i)}=[L(\gc)^{(i)}]^*$, $i\in[\eta_3]$.
\end{definition}

\begin{lemma} {\rm \cite{KK,K3}}\label{ppaa}
Arbitrary tensors $\gc \in \mathbb{C}^{\eta_1\times \eta_2\times \eta_3}$ and $\hc \in \mathbb{C}^{\eta_2 \times \eta_4 \times \eta_3}$ satisfy $(\gc\,\star_{\!\!\! M} \, \hc)^*=\hc^*\,\star_{\!\!\! M} \, \gc^*$.
\end{lemma}

\begin{definition} %\label{d511}
For $\gc \in \CC^{\eta_1\times \eta_1\times \eta_3}$ suppose
$$ %\label{Fa}
\matt(\gc)=\diag\left(L(\gc)^{(1)},\ldots,L(\gc)^{(\eta_3)}\right).
$$
Then, the index of $\gc$ is equal to
$\ind(\gc)=\max\limits_{i\in[\eta_3]}\{\ind(L(\gc)^{(i)})\}$.
\end{definition}

%\begin{definition} %\cite{MSL}\label{dh10}
%%\textbf{Frobenius norm} of $\gc$ is
%\begin{equation*} % \label{se}
%\|\gc\|^2_F = \gc^H\,\star_{\!\!\! M} \, \gc = \sum_{i_1=1}^{\eta_1}\sum_{i_2=1}^{\eta_2}\sum_{i_3=1}^{\eta_3}\gc^2(i_1, i_2, i_3).
%\end{equation*}
%\end{definition}

\section{The most important results}\label{three0}

In this section, the GD inverse of tensors under the M-product is originated and investigated.

\subsection{The GD inverse of a tensor}
Definition \ref{d07} introduces the GD-inverse in tensor case.
\begin{definition}\label{d07}
Let $\gc \in\mathbb{C}^{\eta_1\times \eta_1\times \eta_3}$ and $\ind(\gc)=k$.
If there exists $\xc \in \mathbb{C}^{\eta_1\times \eta_1\times \eta_3}$ which satisfies
\begin{equation}\label{mp012}
\textnormal{(I)} \ \gc\,\star_{\!\!\! M} \, \xc\,\star_{\!\!\! M} \, \gc = \gc, \quad \textnormal{(II)} \ \xc\,\star_{\!\!\! M} \, \gc^{k+1} = \gc^k, \quad \textnormal{(III)} \ \gc^{k+1}\,\star_{\!\!\! M} \, \xc= \gc^k,
\end{equation}
the tensor $\xc$ is called the {\em GD inverse} of $\gc$ and it is marked with $\gc^{GD}$.
\end{definition}

Notice that the G-Drazin inverse is not unique. We use $\gc\{GD\}=\{\gc^{GD}|\gc \in\mathbb{C}^{\eta_1\times \eta_1\times \eta_3}\}$ to denote the set of the GD inverses of a tensor. Next lemma concerns the core-nilpotent decomposition of a tensor.

\begin{lemma}\label{3}

Let $\gc\in\mathbb{C}^{\eta_1\times \eta_1\times \eta_3}$ satisfy $\ind(\gc)=k$.
Then, the core-nilpotent decomposition of $\gc$ is stated as
\begin{equation*}
  \gc=\pc\,\star_{\!\!\! M} \, \matt^{-1}\left[\diag\left(\left[
                               \begin{array}{cc}
                                 U^{(1)} & O \\
                                O  & N^{(1)} \\
                               \end{array}
                             \right],\ldots,\left[
                               \begin{array}{cc}
                                 U^{(\eta_3)} & O \\
                                 O & N^{(\eta_3)} \\
                               \end{array}
                             \right]
 \right) \right] \,\star_{\!\!\! M} \, \pc^{-1},
\end{equation*}
where $\pc$, $U^{(i)}$, $i\in[\eta_3]$ are nonsingular, and $N^{(i)}$, $i\in[\eta_3]$ are nilpotent.

\end{lemma}

\begin{proof}
By (\ref{mmm}), we have
\begin{eqnarray*}
  \gc=\matt^{-1}[\matt(\gc)]=\matt^{-1}\left[\diag\left(
                           (\widehat{\gc})^{(1)},\ldots,
                           (\widehat{\gc})^{(\eta_3)}
                         \right)\right].
\end{eqnarray*}
Suppose the core-nilpotent decomposition of $(\widehat{\gc})^{(i)}$ is
\begin{equation}\label{eew}
(\widehat{\gc})^{(i)}=P^{(i)}\left[
                               \begin{array}{cc}
                                 U^{(i)} & O \\
                                O  & N^{(i)} \\
                               \end{array}
                             \right](P^{(i)})^{-1}, \ \ i\in[\eta_3].
\end{equation}
Then,
\begin{equation*}
\aligned
\gc&=\matt^{-1}\left[\diag\left((\widehat{\gc})^{(1)},\ldots,(\widehat{\gc})^{(\eta_3)}\right)\right]\\
                       &=\matt^{-1}\left[\diag\left(
                            P^{(1)}\left[
                               \begin{array}{cc}
                                 U^{(1)} & O \\
                                O  & N^{(1)} \\
                               \end{array}
                             \right](P^{(1)})^{-1},\ldots,P^{(\eta_3)}\left[
                               \begin{array}{cc}
                                 U^{(\eta_3)} & O \\
                                O  & N^{(\eta_3)} \\
                               \end{array}
                             \right](P^{(\eta_3)})^{-1}\right)\right] \\
&=\matt^{-1}\left[\diag\left(P^{(1)},\ldots,P^{(\eta_3)}\right)\cdot
\diag\left(\left[
                               \begin{array}{cc}
                                 U^{(1)} & O \\
                                O  & N^{(1)} \\
                               \end{array}
                             \right],\ldots,\left[
                               \begin{array}{cc}
                                 U^{(\eta_3)} & O \\
                                O  & N^{(\eta_3)} \\
                               \end{array}
                             \right]\right)\cdot\diag\left(
                 (P^{(1)})^{-1},\ldots,(P^{(\eta_3)})^{-1}\right)\right]\\
&=\pc\,\star_{\!\!\! M} \, \matt^{-1}\left[\diag\left(\left[
                               \begin{array}{cc}
                                 U^{(1)} & O \\
                                O  & N^{(1)} \\
                               \end{array}
                             \right],\ldots,\left[
                               \begin{array}{cc}
                                 U^{(\eta_3)} & O \\
                                O  & N^{(\eta_3)} \\
                               \end{array}
                             \right]\right)\right]
  \,\star_{\!\!\! M} \, \pc^{-1},
  \endaligned
\end{equation*}
where $\pc$, $U^{(i)}$, $i\in[\eta_3]$ are nonsingular, and $N^{(i)}$, $i\in[\eta_3]$ are nilpotent.
\end{proof}

The following theorem is an equivalent statement for the GD inverse.
Here, one of the $\{1\}$-inverse of $G$ or $\gc$ will be denoted by $G^-$ or $\gc^-$.
\begin{theorem}\label{caa}
Let $\gc, \xc\in\mathbb{C}^{\eta_1\times \eta_1\times \eta_3}$ and $\ind(\gc)=k$.
Then, $\xc\in\gc\{GD\}$ if and only if
\begin{equation}\label{101122}
\xc=\pc\,\star_{\!\!\! M} \,
\matt^{-1}\left[\diag\left(
\left[
                               \begin{array}{cc}
                                 (U^{(1)})^{-1} & O \\
                                 O & (N^{(1)})^{-} \\
                               \end{array}
                             \right],\ldots,\left[
                               \begin{array}{cc}
                                 (U^{(\eta_3)})^{-1} & O \\
                                 O & (N^{(\eta_3)})^{-} \\
                               \end{array}
                             \right]
\right)\right]\,\star_{\!\!\! M} \, \pc^{-1},
\end{equation}
where   $(N^{(i)})^{-}\in N^{(i)}\{1\}$, $i\in[\eta_3]$.
\end{theorem}

\begin{proof}
$(\Rightarrow):$ By Lemma \ref{3}, one has
\begin{equation*}
  \gc=\pc\,\star_{\!\!\! M} \, \matt^{-1}\left[\diag\left(\left[
                               \begin{array}{cc}
                                 U^{(1)} & O \\
                                O  & N^{(1)} \\
                               \end{array}
                             \right],\ldots,\left[
                               \begin{array}{cc}
                                 U^{(\eta_3)} & O \\
                                 O & N^{(\eta_3)} \\
                               \end{array}
                             \right]
 \right)\right] \,\star_{\!\!\! M} \, \pc^{-1}.
\end{equation*}
In order to establish the formula $({\ref{101122}})$, consider
\begin{equation*}
  \xc=\matt^{-1}\left[\diag\left( (\widehat{\xc})^{(1)},\ldots,(\widehat{\xc})^{(\eta_3)}\right)\right]=
  \pc\,\star_{\!\!\! M} \, \matt^{-1}\left[\diag\left(\left[
                               \begin{array}{cc}
                               S^{(1)}& L^{(1)} \\
C^{(1)}& D^{(1)}
                               \end{array}
                             \right],\ldots,\left[
                               \begin{array}{cc}
                                S^{(\eta_3)}& L^{(\eta_3)} \\
C^{(\eta_3)}& D^{(\eta_3)}
                               \end{array}
                             \right]
\right) \right]  \,\star_{\!\!\! M} \, \pc^{-1}.
\end{equation*}
In this case,
$$
(\widehat{\xc})^{(i)}=P^{(i)}
\begin{bmatrix}
S^{(i)}& L^{(i)} \\
C^{(i)}& D^{(i)}
\end{bmatrix}(P^{(i)})^{-1},~i\in[\eta_3].$$
By $\textnormal{(I)}$ of Definition \ref{d07}, we obtain $\matt(\gc)\matt(\xc)\matt(\gc)=\matt(\gc)$, and further
\begin{align*}
  (\widehat{\gc})^{(i)}(\widehat{\xc})^{(i)}(\widehat{\gc})^{(i)}&=P^{(i)}
\begin{bmatrix}
U^{(i)}S^{(i)}U^{(i)}&U^{(i)}L^{(i)}N^{(i)}\\
N^{(i)}C^{(i)}U^{(i)}&N^{(i)}D^{(i)}N^{(i)}
\end{bmatrix}(P^{(i)})^{-1}\\
&=P^{(i)}
\begin{bmatrix}
U^{(i)}&O\\
O&N^{(i)}
\end{bmatrix}(P^{(i)})^{-1}=(\widehat{\gc})^{(i)},~i\in[\eta_3].
\end{align*}
Then it is concluded that $U^{(i)}L^{(i)}N^{(i)}=O$, $N^{(i)}C^{(i)}U^{(i)}=O$, $S^{(i)}=(U^{(i)})^{-1}$ and $D^{(i)}=(N^{(i)})^{-}\in N^{(i)}\{1\}$.\\ By $\textnormal{(II)}$ of Definition \ref{d07}, we obtain $\matt(\xc)\matt(\gc^{k+1})=\matt(\gc^{k})$, then
$$(\widehat{\xc})^{(i)}(\widehat{\gc^{k+1}})^{(i)}=P^{(i)}
\begin{bmatrix}
S^{(i)}(U^{k+1})^{(i)}&O\\
C^{(i)}(U^{k+1})^{(i)}&O
\end{bmatrix}(P^{(i)})^{-1}=P^{(i)}
\begin{bmatrix}
 (U^{k})^{(i)}&O\\
O&O
\end{bmatrix}(P^{(i)})^{-1}=(\widehat{\gc^k})^{(i)},~i\in[\eta_3].
$$
In this case, $C^{(i)}=O$. By $\textnormal{(III)}$ of Definition \ref{d07}, we obtain $\matt(\gc^{k+1})\matt(\xc)=\matt(\gc^k)$, then
$$(\widehat{\gc^{k+1}})^{(i)}(\widehat{\xc})^{(i)}=P^{(i)}
\begin{bmatrix}
 (U^{k+1})^{(i)}S^{(i)}&(U^{k+1})^{(i)}L^{(i)}\\
O&O
\end{bmatrix}(P^{(i)})^{-1}=P^{(i)}
\begin{bmatrix}
 (U^{k})^{(i)}&O\\
O&O
\end{bmatrix}(P^{(i)})^{-1}=(\widehat{\gc^k})^{(i)},~i\in[\eta_3],$$
which implies $L^{(i)}=0$. Hence,
\begin{equation*}
  (\widehat{\xc})^{(i)}=P^{(i)}\left[
                               \begin{array}{cc}
                                 (U^{-1})^{(i)} & O \\
                                O  & (N^{(i)})^{-} \\
                               \end{array}
                             \right](P^{(i)})^{-1},
\end{equation*}
and then
\begin{equation*}
\xc=\pc\,\star_{\!\!\! M} \,
\matt^{-1}\left[\diag\left(
\left[
                               \begin{array}{cc}
                                 (U^{(1)})^{-1} & O \\
                                 O & (N^{(1)})^{-} \\
                               \end{array}
                             \right],\ldots,\left[
                               \begin{array}{cc}
                                 (U^{(\eta_3)})^{-1} & O \\
                                 O & (N^{(\eta_3)})^{-} \\
                               \end{array}
                             \right]
\right)\right]\,\star_{\!\!\! M} \, \pc^{-1}.
\end{equation*}
$(\Leftarrow):$ Let
\begin{equation*}
  \gc=\matt^{-1}\left[\diag\left( (\widehat{\gc})^{(1)},\ldots,(\widehat{\gc})^{(\eta_3)})\right)\right]=
  \pc\,\star_{\!\!\! M} \, \matt^{-1}\left[\diag\left(\left[
                               \begin{array}{cc}
                                 U^{(1)} & O \\
                                O  & N^{(1)} \\
                               \end{array}
                             \right],\ldots,\left[
                               \begin{array}{cc}
                                 U^{(\eta_3)} & O \\
                                 O & N^{(\eta_3)} \\
                               \end{array}
                             \right]
\right) \right]  \,\star_{\!\!\! M} \, \pc^{-1}.
\end{equation*}
Since \begin{equation*}\label{1012}
\xc=\pc\,\star_{\!\!\! M} \,
\matt^{-1}\left[\diag\left(
\left[
                               \begin{array}{cc}
                                 (U^{(1)})^{-1} & O \\
                                 O & (N^{(1)})^{-} \\
                               \end{array}
                             \right],\ldots,\left[
                               \begin{array}{cc}
                                 (U^{(\eta_3)})^{-1} & O \\
                                 O & (N^{(\eta_3)})^{-} \\
                               \end{array}
                             \right]
\right)\right]\,\star_{\!\!\! M} \, \pc^{-1},
\end{equation*}
it is easy to check
\begin{equation*}\label{mp012}
\textnormal{(I)} \ \gc\,\star_{\!\!\! M} \, \xc\,\star_{\!\!\! M} \, \gc = \gc, \quad \textnormal{(II)} \ \xc\,\star_{\!\!\! M} \, \gc^{k+1} = \gc^k, \quad \textnormal{(III)} \ \gc^{k+1}\,\star_{\!\!\! M} \, \xc= \gc^k.
\end{equation*}
Hence, $\xc\in\gc\{GD\}$.
\end{proof}

Algorithm \ref{Alg1MT} for computing the GD inverse is constructed based on Theorem \ref{caa}.

\begin{algorithm}[H]\label{Alg1MT}
\caption{Computing the GD inverse under the M-product}	
\KwIn{$\gc\in\mathbb{C}^{\eta_1\times \eta_1\times \eta_3}$ and $M \in\mathbb{C}^{\eta_3\times \eta_3}$}
\KwOut {$\xc=\gc^{GD}$}
\begin{enumerate}\label{al}
\item Compute $\widehat{\gc}=\gc\times_{\!\!3}M$
\item $k=\ind(\widehat{\gc})$
\item $\textbf{For}$ $i=1:\eta_3$ \textbf{do}
\item ~~~~~~$(\widehat{\pc})^{(i)}=\lbrack X\left|Y\right.\rbrack$ (X is a  base of $R(((\widehat{\gc})^{(i)})^{k})$,~Y is a  base of $N(((\widehat{\gc})^{(i)})^{k})$
\textbf{)}
\item ~~~~~~$(\widehat{\tc})^{(i)}=((\widehat{\mathcal {P}})^{(i)})^{-1}(\widehat{\gc})^{(i)}(\widehat{\mathcal {P}})^{(i)}$
\item ~~~~~~$U=(\widehat{\tc})^{(i)}(1:k;1:k)$, $N=(\widehat{\tc})^{(i)}(k+1:\eta_3;k+1:\eta_3)$
\item ~~~~~~$(\widehat{\dc})^{(i)}=\mathrm{diag}(U^{-1},N^{(1)})$
\item ~~~~~~$(\widehat{\gc^{GD}})^{(i)}=(\widehat{\mathcal {P}})^{(i)}(\widehat{\mathcal {D}})^{(i)}((\widehat{\mathcal {P}})^{(i)})^{-1}$
 \item ~~~~~~$\zc^{(i)}=(\widehat{\gc^{GD}})^{(i)}$
\item \textbf{End for}
\item Calculate $\xc=\zc\times_3M^{-1}$

\item \textbf{Return} $\gc^{GD}=\xc$.
\end{enumerate}
\end{algorithm}

\begin{example}
Let $\gc\in\mathbb{C}^{3\times3\times3}$ and $M\in\mathbb{C}^{3\times3}$  with entries
$$(\mathcal A)^{(1)}=
\begin{bmatrix}
1&2&5\\
3&3&5\\
1&2&4
\end{bmatrix},~
(\mathcal A)^{(2)}=
\begin{bmatrix}
1&2&3\\
1&2&2\\
1&2&3
\end{bmatrix},~
(\mathcal A)^{(3)}=
\begin{bmatrix}
-1&0&1\\
0&0&1\\
-1&0&1
\end{bmatrix},~
M=
\begin{bmatrix}
1&0&-1\\
0&1&0\\
0&1&-1
\end{bmatrix}.
$$
The  index of $\gc$ is $k=2$, since $\ind((\widehat{\gc})^{(1)})=\ind((\widehat{\gc})^{(2)})=\ind((\widehat{\gc})^{(3)})=2$.
By Algorithm \ref{al}, compute $\xc=\gc^{GD}$, which gives
$$
(\mathcal X)^{(1)}=
\begin{bmatrix}
-14.75&-14.5&37.25\\
12.25&12&-30.25\\
1.25&1.5&-3.75
\end{bmatrix},~
(\mathcal X)^{(2)}=
\begin{bmatrix}
-0.5&-1&1.5\\
1&2&-2.5\\
-0.5&-1&1.5
\end{bmatrix},~
(\mathcal X)^{(3)}=
\begin{bmatrix}
-0.75&-0.5&1.25\\
1.25&1&-2.25\\
-0.75&-0.5&1.25
\end{bmatrix}.
$$
\end{example}

Theorem \ref{ty2} gives another equivalent condition for the tensor GD inverse.

\begin{theorem}\label{ty2}
Consider $\gc, \xc\in\mathbb{C}^{\eta_1\times \eta_1\times \eta_3}$ such that $\ind(\gc)=k$. Then, $\xc\in\gc\{GD\}$ if and only if
\begin{equation}\label{cca}
\textnormal{(I)} \ \gc\,\star_{\!\!\! M} \, \xc\,\star_{\!\!\! M} \, \gc = \gc, \quad \textnormal{(II)} \ \gc^{k}\,\star_{\!\!\! M} \, \xc =\xc\,\star_{\!\!\! M} \, \gc^k.
\end{equation}
\end{theorem}

\begin{proof}
By Theorem ${\ref{caa}}$, it is enough to check
\begin{eqnarray*}\label{10122}
\xc
=\pc\,\star_{\!\!\! M} \,
\matt^{-1}\left[\diag\left(
\left[
                               \begin{array}{cc}
                                 (U^{(1)})^{-1} & O \\
                                 O & (N^{(1)})^{-} \\
                               \end{array}
                             \right],\ldots,\left[
                               \begin{array}{cc}
                                 (U^{(\eta_3)})^{-1} & O \\
                                 O & (N^{(\eta_3)})^{-} \\
                               \end{array}
                             \right]
\right)\right]\,\star_{\!\!\! M} \, \pc^{-1}
\end{eqnarray*}
is equivalent to
\begin{equation*}\label{cca}
\textnormal{(I)} \ \gc\,\star_{\!\!\! M} \, \xc\,\star_{\!\!\! M} \, \gc = \gc, \quad \textnormal{(II)} \ \gc^{k}\,\star_{\!\!\! M} \, \xc =\xc\,\star_{\!\!\! M} \, \gc^k.
\end{equation*}
Suppose
\begin{equation*}
  \gc=\matt^{-1}\left[\diag\left((\widehat{\gc})^{(1)},\ldots,(\widehat{\gc})^{(\eta_3)}\right)\right]=
  \pc\,\star_{\!\!\! M} \, \matt^{-1}\left[\diag\left(\left[
                               \begin{array}{cc}
                                 U^{(1)} & O \\
                                O  & N^{(1)} \\
                               \end{array}
                             \right],\ldots,\left[
                               \begin{array}{cc}
                                 U^{(\eta_3)} & O \\
                                 O & N^{(\eta_3)} \\
                               \end{array}
                             \right]
\right) \right]  \,\star_{\!\!\! M} \, \pc^{-1},
\end{equation*}
it is simple to verify $\gc\,\star_{\!\!\! M} \, \xc\,\star_{\!\!\! M} \, \gc = \gc$ and $\gc^{k}\,\star_{\!\!\! M} \, \xc=\xc\,\star_{\!\!\! M} \, \gc^{k}$.

Now, we will prove the reverse part. Suppose
\begin{equation*}
  \xc=\matt^{-1}\left[\diag\left( (\widehat{\xc})^{(1)},\ldots,(\widehat{\xc})^{(\eta_3)}\right)\right]=
  \pc\,\star_{\!\!\! M} \, \matt^{-1}\left[\diag\left(\left[
                               \begin{array}{cc}
                               S^{(1)}& L^{(1)} \\
C^{(1)}& D^{(1)}
                               \end{array}
                             \right],\cdots,\left[
                               \begin{array}{cc}
                                S^{(\eta_3)}& L^{(\eta_3)} \\
C^{(\eta_3)}& D^{(\eta_3)}
                               \end{array}
                             \right]
 \right) \right] \,\star_{\!\!\! M} \, \pc^{-1}.
\end{equation*}
Consider the partition
$$(\widehat{\xc})^{(i)}=P^{(i)}
\begin{bmatrix}
S^{(i)}& L^{(i)} \\
C^{(i)}& D^{(i)}
\end{bmatrix}(P^{(i)})^{-1},~i\in[\eta_3].$$
By $\gc^{k}\,\star_{\!\!\! M} \, \xc =\xc\,\star_{\!\!\! M} \, \gc^k$, we have $\matt(\gc^k)\matt(\xc)=\matt(\xc)\matt(\gc^k)$, which implies
\begin{align*}
  (\widehat{\gc^k})^{(i)}(\widehat{\xc})^{(i)}&=P^{(i)}
\begin{bmatrix}
(U^{k})^{(i)}S^{(i)}&(U^{k})^{(i)}L^{(i)}\\
O&O
\end{bmatrix}(P^{(i)})^{-1}\\
&=P^{(i)}
\begin{bmatrix}
S^{(i)}(U^{k})^{(i)}&O\\
C^{(i)}(U^{k})^{(i)}&O
\end{bmatrix}(P^{(i)})^{-1}=(\widehat{\xc})^{(i)}(\widehat{\gc^{k}})^{(i)}.
\end{align*}
Hence, we obtain $L^{(i)}=O$ and $C^{(i)}=O$, $i\in[\eta_3]$.
On the other hand, with $\gc\,\star_{\!\!\! M} \, \xc\,\star_{\!\!\! M} \, \gc=\gc$, we get $\matt(\gc)\matt(\xc)\matt(\gc)=\matt(\gc)$. Thus,
\begin{align*}
(\widehat{\gc})^{(i)}(\widehat{\xc})^{(i)}(\widehat{\gc})^{(i)}
&=P^{(i)}
\begin{bmatrix}
 U^{(i)}S^{(i)}U^{(i)}  &  O\\
O                    &  N^{(i)}D^{(i)}N^{(i)}
\end{bmatrix}(P^{(i)})^{-1}\\
&=P^{(i)}
\begin{bmatrix}
 U^{(i)}  &  0\\
0                  &   N^{(i)}
\end{bmatrix}(P^{(i)})^{-1}=(\widehat{\gc})^{(i)},~i\in[\eta_3].
\end{align*}
So, we have $S^{(i)}=(U^{(i)})^{-1}$ and $D^{(i)}=(N^{(i)})^{-}\in N^{(i)}\{1\}$.
Therefore, $\xc\in\gc\{GD\}$.
\end{proof}

Definition \ref{DefDIMT} introduces the tensor Drazin inverse under the M-product.

\begin{definition}\label{DefDIMT}
Let $\gc, \xc\in\mathbb{C}^{\eta_1\times \eta_1\times \eta_3}$ is of index $\ind(\gc)=k$.
In this case, $\xc$ denoted  by $\gc^D$, is called the Drazin inverse of $\gc$ if
\begin{equation*}
  \textnormal{(I)} \ \xc\,\star_{\!\!\! M} \, \gc^{k+1}=\gc^{k}, \ \ \textnormal{(II)} \ \xc\,\star_{\!\!\! M} \, \gc\,\star_{\!\!\! M} \, \xc=\xc, \ \ \textnormal{(III)} \ \xc\,\star_{\!\!\! M} \, \gc=\gc\,\star_{\!\!\! M} \, \xc.
\end{equation*}
\end{definition}

\begin{lemma}\label{4}
Let $\gc, \xc\in\mathbb{C}^{\eta_1\times \eta_1\times \eta_3}$ and $\ind(\gc)=k$.
Then,
\begin{equation*}
  \gc^D=\pc\,\star_{\!\!\! M} \, \matt^{-1}\left[\diag\left(\left[
                               \begin{array}{cc}
                                 (U^{(1)})^{-1} & O \\
                                O  &O \\
                               \end{array}
                             \right],\cdots,\left[
                               \begin{array}{cc}
                                 (U^{(\eta_3)})^{-1}& O \\
                                 O & O \\
                               \end{array}
                             \right]
\right) \right]  \,\star_{\!\!\! M} \, \pc^{-1}.
\end{equation*}
\end{lemma}

\begin{proof}
By Lemma \ref{3}, we have
\begin{equation*}
  \gc=\pc\,\star_{\!\!\! M} \, \matt^{-1}\left[\diag\left(\left[
                               \begin{array}{cc}
                                 U^{(1)} & O \\
                                O  & N^{(1)} \\
                               \end{array}
                             \right],\ldots,\left[
                               \begin{array}{cc}
                                 U^{(\eta_3)} & O \\
                                 O & N^{(\eta_3)} \\
                               \end{array}
                             \right]
                           \right)  \right]
  \,\star_{\!\!\! M} \, \pc^{-1},
\end{equation*}
Based on the  Drazin property, we get
\begin{equation*}
  \gc^D=\pc\,\star_{\!\!\! M} \, \matt^{-1}\left[\diag\left(\left[
                               \begin{array}{cc}
                                 (U^{(1)})^{-1} & O \\
                                O  &O \\
                               \end{array}
                             \right],\ldots,\left[
                               \begin{array}{cc}
                                 (U^{(\eta_3)})^{-1}& O \\
                                 O & O \\
                               \end{array}
                             \right]
 \right) \right] \,\star_{\!\!\! M} \, \pc^{-1}.
\end{equation*}
\end{proof}

\begin{theorem}\label{ass}
Let $\gc, \xc\in\mathbb{C}^{\eta_1\times \eta_1\times \eta_3}$ and $\ind(\gc)=k$. Then, $\xc\in\gc\{GD\}$ if and only if
\begin{equation}\label{cca2}
\textnormal{(I)} \ \gc\,\star_{\!\!\! M} \, \xc\,\star_{\!\!\! M} \, \gc = \gc, \quad \textnormal{(II)} \ \gc^D\,\star_{\!\!\! M} \, \gc\,\star_{\!\!\! M} \, \xc =\xc\,\star_{\!\!\! M} \, \gc^D\,\star_{\!\!\! M} \, \gc.
\end{equation}
\end{theorem}

\begin{proof}
Following results of Theorem \ref{caa}, it is enough to verify
\begin{equation*}\label{10122}
\xc=\pc\,\star_{\!\!\! M} \,
\matt^{-1}\left[\diag\left(
\left[
                               \begin{array}{cc}
                                 (U^{(1)})^{-1} & O \\
                                 O & (N^{(1)})^{-} \\
                               \end{array}
                             \right],\ldots,\left[
                               \begin{array}{cc}
                                 (U^{(\eta_3)})^{-1} & O \\
                                 O & (N^{(\eta_3)})^{-} \\
                               \end{array}
                             \right]
\right)\right]\,\star_{\!\!\! M} \, \pc^{-1}
\end{equation*}
is equivalent to
\begin{equation*}\label{fff}
\textnormal{(I)} \ \gc\,\star_{\!\!\! M} \, \xc\,\star_{\!\!\! M} \, \gc = \gc, \quad \textnormal{(II)} \ \gc^D\,\star_{\!\!\! M} \, \gc\,\star_{\!\!\! M} \, \xc =\xc\,\star_{\!\!\! M} \, \gc^D\,\star_{\!\!\! M} \, \gc.
\end{equation*}
It is simple to check $\gc\,\star_{\!\!\! M} \, \xc\,\star_{\!\!\! M} \, \gc = \gc$ and $\gc^D\,\star_{\!\!\! M} \, \gc\,\star_{\!\!\! M} \, \xc =\xc\,\star_{\!\!\! M} \, \gc^D\,\star_{\!\!\! M} \, \gc$, while $\xc$ possesses the above form. On the other hand, Lemma \ref{4} implies
$$\gc^D=\pc\,\star_{\!\!\! M} \, \matt^{-1}\left[\diag\left(\left[
                               \begin{array}{cc}
                                 (U^{(1)})^{-1} & O \\
                                O  &O \\
                               \end{array}
                             \right],\ldots,\left[
                               \begin{array}{cc}
                                 (U^{(\eta_3)})^{-1}& O \\
                                 O & O \\
                               \end{array}
                             \right]
 \right)\right] \,\star_{\!\!\! M} \, \pc^{-1}.
$$
Now, we will prove the reverse part. Suppose
\begin{equation*}
  \xc=\matt^{-1}\left[\diag\left(\widehat{\xc})^{(1)},\ldots,(\widehat{\xc})^{(\eta_3)}\right)\right]=
  \pc\,\star_{\!\!\! M} \, \matt^{-1}\left[\diag\left(\left[
                               \begin{array}{cc}
                               S^{(1)}& L^{(1)} \\
C^{(1)}& D^{(1)}
                               \end{array}
                             \right],\ldots,\left[
                               \begin{array}{cc}
                                S^{(\eta_3)}& L^{(\eta_3)} \\
C^{(\eta_3)}& D^{(\eta_3)}
                               \end{array}
                             \right]
 \right)\right] \,\star_{\!\!\! M} \, \pc^{-1}.
\end{equation*}
Consider the partition
$$(\widehat{\xc})^{(i)}=P^{(i)}
\begin{bmatrix}
S^{(i)}& L^{(i)} \\
C^{(i)}& D^{(i)}
\end{bmatrix}(P^{(i)})^{-1}
~\text{and}~
(\widehat{\gc^D})^{(i)}=P^{(i)}\left[
\begin{array}{cc}
                                 (U^{(i)})^{-1} & O \\
                                O  &O \\
                               \end{array} \right]
(P^{(i)})^{-1}
,~i\in[\eta_3].$$
It follows from $\gc^D\,\star_{\!\!\! M} \, \gc\,\star_{\!\!\! M} \, \xc =\xc\,\star_{\!\!\! M} \, \gc^D\,\star_{\!\!\! M} \, \gc$ that $\matt(\gc^D)\matt(\gc)\matt(\xc)=\matt(\xc)\matt(\gc^D)\matt(\gc)$, which further implies
\begin{align*}
 (\widehat{\gc^D})^{(i)}(\widehat{\gc})^{(i)}(\widehat{\xc})^{(i)}=
P^{(i)}
\begin{bmatrix}
S^{(i)} & L^{(i)} \\
O &O
\end{bmatrix} (P^{(i)})^{-1}=P^{(i)}
\begin{bmatrix}
S^{(i)} & O \\
C^{(i)} &O
\end{bmatrix}(P^{(i)})^{-1}=(\widehat{\xc})^{(i)}(\widehat{\gc^D})^{(i)}(\widehat{\gc})^{(i)}.
\end{align*}
Hence, we obtain $L^{(i)}=O$ and $C^{(i)}=O$, $i\in[\eta_3]$.
On the other hand, with $\gc\,\star_{\!\!\! M} \, \xc\,\star_{\!\!\! M} \, \gc=\gc$, we get $\matt(\gc)\matt(\xc)\matt(\gc)=\matt(\gc)$. Thus,
\begin{align*}
(\widehat{\gc})^{(i)}(\widehat{\xc})^{(i)}(\widehat{\gc})^{(i)}
=P^{(i)}
\begin{bmatrix}
 U^{(i)}S^{(i)}U^{(i)}  &  O\\
O                    &  N^{(i)}D^{(i)}N^{(i)}
\end{bmatrix}(P^{(i)})^{-1}
=P^{(i)}
\begin{bmatrix}
 U^{(i)}  &  0\\
0                  &   N^{(i)}
\end{bmatrix}(P^{(i)})^{-1},~i\in[\eta_3].
\end{align*}
So, identities $S^{(i)}=(U^{(i)})^{-1}$ and $D^{(i)}=(N^{(i)})^{-}\in N^{(i)}\{1\}$ are valid. Therefore, we conclude $\xc\in\gc\{GD\}$.
\end{proof}

\begin{theorem}
Assume $\gc, \xc\in\mathbb{C}^{\eta_1\times \eta_1\times \eta_3}$ and $\ind(\gc)=k$.
In this case $\xc\in\gc\{GD\}$ if and only if
\begin{align*}
\textnormal{(I)} \  \gc\,\star_{\!\!\! M} \, \xc\,\star_{\!\!\! M} \, \gc = \gc, \quad \textnormal{(II)} \ \gc^D\,\star_{\!\!\! M} \, \gc^2\,\star_{\!\!\! M} \, \xc=\gc\,\star_{\!\!\! M} \, \gc^D=\xc\,\star_{\!\!\! M} \, \gc^D\,\star_{\!\!\! M} \, \gc^2.
\end{align*}
\end{theorem}

\begin{proof}
By Theorem ${\ref{caa}}$, it is enough to prove
\begin{equation*}\label{10122}
\xc=\pc\,\star_{\!\!\! M} \,
\matt^{-1}\left[\diag\left(
\left[
                               \begin{array}{cc}
                                 (U^{(1)})^{-1} & O \\
                                 O & (N^{(1)})^{-} \\
                               \end{array}
                             \right],\ldots,\left[
                               \begin{array}{cc}
                                 (U^{(\eta_3)})^{-1} & O \\
                                 O & (N^{(\eta_3)})^{-} \\
                               \end{array}
                             \right]
\right)\right]\,\star_{\!\!\! M} \, \pc^{-1}
\end{equation*}
is equivalent to
\begin{equation*}\label{fff}
\textnormal{(I)} \ \gc\,\star_{\!\!\! M} \, \xc\,\star_{\!\!\! M} \, \gc = \gc, \quad \textnormal{(II)} \ \gc^D\,\star_{\!\!\! M} \, \gc^2\,\star_{\!\!\! M} \, \xc=\gc\,\star_{\!\!\! M} \, \gc^D=\xc\,\star_{\!\!\! M} \, \gc^D\,\star_{\!\!\! M} \, \gc^2.
\end{equation*}
It is simple to check $\gc\,\star_{\!\!\! M} \, \xc\,\star_{\!\!\! M} \, \gc = \gc$ and $\gc^D\,\star_{\!\!\! M} \, \gc^2\,\star_{\!\!\! M} \, \xc=\gc\,\star_{\!\!\! M} \, \gc^D=\xc\,\star_{\!\!\! M} \, \gc^D\,\star_{\!\!\! M} \, \gc^2$ for $\xc$ of the above form.
Now, by Lemma \ref{4}, conclusion is
$$
  \gc^D=\pc\,\star_{\!\!\! M} \, \matt^{-1}\left[\diag\left(\left[
                               \begin{array}{cc}
                                 (U^{(1)})^{-1} & O \\
                                O  &O \\
                               \end{array}
                             \right],\ldots,\left[
                               \begin{array}{cc}
                                 (U^{(\eta_3)})^{-1}& O \\
                                 O & O \\
                               \end{array}
                             \right]\right) \right]
  \,\star_{\!\!\! M} \, \pc^{-1}.
$$
Next, we will prove the reverse part. Suppose
\begin{equation*}
  \xc=\matt^{-1}\left[\diag\left( (\widehat{\xc})^{(1)},\cdots,(\widehat{\xc})^{(\eta_3)}\right)\right]=
  \pc\,\star_{\!\!\! M} \, \matt^{-1}\left[\diag\left(\left[
                               \begin{array}{cc}
                               S^{(1)}& L^{(1)} \\
C^{(1)}& D^{(1)}
                               \end{array}
                             \right],\ldots,\left[
                               \begin{array}{cc}
                                S^{(\eta_3)}& L^{(\eta_3)} \\
C^{(\eta_3)}& D^{(\eta_3)}
                               \end{array}\right]
                           \right)  \right]
  \,\star_{\!\!\! M} \, \pc^{-1}.
\end{equation*}
Consider the partition
$$(\widehat{\xc})^{(i)}=P^{(i)}
\begin{bmatrix}
S^{(i)}& L^{(i)} \\
C^{(i)}& D^{(i)}
\end{bmatrix}(P^{(i)})^{-1}
~\text{and}~
(\widehat{\gc^D})^{(i)}=P^{(i)}\left[
\begin{array}{cc}
                                 (U^{(i)})^{-1} & O \\
                                O  &O \\
                               \end{array} \right]
(P^{(i)})^{-1}
,~i\in[\eta_3].$$
With $\gc\,\star_{\!\!\! M} \, \xc\,\star_{\!\!\! M} \, \gc=\gc$, we get $\matt(\gc)\matt(\xc)\matt(\gc)=\matt(\gc)$. Thus,
\begin{align*}
(\widehat{\gc})^{(i)}(\widehat{\xc})^{(i)}(\widehat{\gc})^{(i)}
=P^{(i)}
\begin{bmatrix}
 U^{(i)}S^{(i)}U^{(i)}  &  U^{(i)}L^{(i)}N^{(i)}\\
N^{(i)}C^{(i)}U^{(i)}                   &  N^{(i)}D^{(i)}N^{(i)}
\end{bmatrix}(P^{(i)})^{-1}=P^{(i)}
\begin{bmatrix}
 U^{(i)}  &  O\\
O                  &   N^{(i)}
\end{bmatrix}(P^{(i)})^{-1},~i\in[\eta_3].
\end{align*}
So, we have $S^{(i)}=(U^{(i)})^{-1}$ and $D^{(i)}=(N^{(i)})^{-}\in N^{(i)}\{1\}$.

On the other hand, by $\gc^D\,\star_{\!\!\! M} \, \gc^2\,\star_{\!\!\! M} \, \xc=\gc\,\star_{\!\!\! M} \, \gc^D=\xc\,\star_{\!\!\! M} \, \gc^D\,\star_{\!\!\! M} \, \gc^2$,
we get
$$\matt(\gc^D)\matt(\gc^2)\matt(\xc)=\matt(\gc)\matt(\gc^D)=\matt(\xc)\matt(\gc^D)\matt(\gc^2),$$
which implies
$$(\widehat{\gc^D})^{(i)}(\widehat{\gc^2})^{(i)}(\widehat{\xc})^{(i)}=P^{(i)}
\begin{bmatrix}
U^{(i)}S^{(i)} & U^{(i)}L^{(i)} \\
O &O
\end{bmatrix}(P^{-1})^{(i)}
=P^{(i)}
\begin{bmatrix}
I & O \\
O &O
\end{bmatrix}(P^{-1})^{(i)}=(\widehat{\gc})^{(i)}(\widehat{\gc^D})^{(i)}$$
and
$$(\widehat{\xc})^{(i)}(\widehat{\gc^D})^{(i)}(\widehat{\gc^2})^{(i)}=P^{(i)}
\begin{bmatrix}
S^{(i)}U^{(i)} &O \\
C^{(i)}U^{(i)} &O
\end{bmatrix}(P^{-1})^{(i)}
=P^{(i)}
\begin{bmatrix}
I & O \\
O &O
\end{bmatrix}(P^{-1})^{(i)}=(\widehat{\gc})^{(i)}(\widehat{\gc^D})^{(i)}.$$
Hence, we obtain $L^{(i)}=O$ and $C^{(i)}=O$, $i\in[\eta_3]$.
 Therefore, we conclude $\xc\in\gc\{GD\}$.
\end{proof}

The following three theorems are the reverse order law of the GD inverse.

\begin{theorem}\label{T5}
Let $\gc, \hc\in\mathbb{C}^{\eta_1\times \eta_1\times \eta_3}$ and $\gc\,\star_{\!\!\! M} \, \hc^2=\hc^2\,\star_{\!\!\! M} \, \gc=\hc\,\star_{\!\!\! M} \, \gc\,\star_{\!\!\! M} \, \hc$. Suppose $k=\max\{\ind(\gc), \ind(\hc)\}$. If
$$\gc\,\star_{\!\!\! M} \, \hc\,\star_{\!\!\! M} \, \hc^{GD}=\hc\,\star_{\!\!\! M} \, \hc^{GD}\,\star_{\!\!\! M} \, \gc, \ \ \   \hc\,\star_{\!\!\! M} \, \hc^{GD}\,\star_{\!\!\! M} \, \gc^{GD}=\gc^{GD}\,\star_{\!\!\! M} \, \hc\,\star_{\!\!\! M} \, \hc^{GD},$$
then $(\gc\,\star_{\!\!\! M} \, \hc)^{GD}=\hc^{GD}\,\star_{\!\!\! M} \, \gc^{GD}$.
\begin{proof}
If $\gc\,\star_{\!\!\! M} \, \hc^2=\hc^2\,\star_{\!\!\! M} \, \gc=\hc\,\star_{\!\!\! M} \, \gc\,\star_{\!\!\! M} \, \hc$, then $${(\gc\,\star_{\!\!\! M} \, \hc)}^k=\gc^k\,\star_{\!\!\! M} \, \hc^k=\hc^k\,\star_{\!\!\! M} \, \gc^k, \ \gc^{k+1}\,\star_{\!\!\! M} \, \hc^k=\hc^k\,\star_{\!\!\! M} \, \gc^{k+1}, \ \gc^k\,\star_{\!\!\! M} \, \hc^{k+1}=\hc^{k+1}\,\star_{\!\!\! M} \, \gc^k.$$
Now, one has
$$\gc\,\star_{\!\!\! M} \, \hc\,\star_{\!\!\! M} \, \hc^{GD}\,\star_{\!\!\! M} \, \gc^{GD}\,\star_{\!\!\! M} \, \gc\,\star_{\!\!\! M} \, \hc=
\gc\,\star_{\!\!\! M} \, \gc^{GD}\,\star_{\!\!\! M} \, \hc\star\hc^{GD}\,\star_{\!\!\! M} \, \gc\,\star_{\!\!\! M} \, \hc=\gc\,\star_{\!\!\! M} \, \gc^{GD}\,\star_{\!\!\! M} \, \gc\,\star_{\!\!\! M} \, \hc\,\star_{\!\!\! M} \, \hc^{GD}\,\star_{\!\!\! M} \, \hc=\gc\,\star_{\!\!\! M} \, \hc.$$
Furthermore, it follows
\begin{align*}
  \hc^{GD}\,\star_{\!\!\! M} \, \gc^{GD}\,\star_{\!\!\! M} \, (\gc\,\star_{\!\!\! M} \, \hc)^{k+1}&= \hc^{GD}\,\star_{\!\!\! M} \, \gc^{GD}\,\star_{\!\!\! M} \, \gc^{k+1}\,\star_{\!\!\! M} \, \hc^{k+1}
  = \hc^{GD}\,\star_{\!\!\! M} \, \gc^{k}\,\star_{\!\!\! M} \, \hc^{k+1}\\
  &= \hc^{GD}\,\star_{\!\!\! M} \, \hc^{k+1}\,\star_{\!\!\! M} \, \gc^{k}
  = \hc^k\,\star_{\!\!\! M} \, \gc^k
  =\gc^k\,\star_{\!\!\! M} \, \hc^k
  =(\gc\,\star_{\!\!\! M} \, \hc)^k,
\end{align*}
as well as
\begin{align*}
 (\gc\,\star_{\!\!\! M} \, \hc)^{k+1}\,\star_{\!\!\! M} \, \hc^{GD}\,\star_{\!\!\! M} \, \gc^{GD}&=\gc^{k+1}\,\star_{\!\!\! M} \,
 \hc^{k+1}\,\star_{\!\!\! M} \, \hc^{GD}\,\star_{\!\!\! M} \, \gc^{GD}
 =\gc^{k+1}\,\star_{\!\!\! M} \, \hc^{k}\,\star_{\!\!\! M} \, \gc^{GD}\\
 &=\hc^{k}\,\star_{\!\!\! M} \, \gc^{k+1}\,\star_{\!\!\! M} \, \gc^{GD}
 =\hc^{k}\,\star_{\!\!\! M} \, \gc^{k}
 =(\gc\,\star_{\!\!\! M} \, \hc)^k.
\end{align*}
Using Theorem \ref{ty2}, it clearly shows that $(\gc\,\star_{\!\!\! M} \, \hc)^{GD}=\hc^{GD}\,\star_{\!\!\! M} \, \gc^{GD}$.
\end{proof}
\end{theorem}

\begin{theorem} \label{T6}
Let $\gc, \hc\in\mathbb{C}^{\eta_1\times \eta_1\times \eta_3}$ and $k=\max\{\ind(\gc), \ind(\hc)\}$. If
$$\gc\,\star_{\!\!\! M} \, \hc=\hc\,\star_{\!\!\! M} \, \gc, \ \ \  \hc\,\star_{\!\!\! M} \, \hc^{GD}\,\star_{\!\!\! M} \, \gc^{GD}=\gc^{GD}\,\star_{\!\!\! M} \, \hc\,\star_{\!\!\! M} \, \hc^{GD},$$
then $(\gc\,\star_{\!\!\! M} \, \hc)^{GD}=\hc^{GD}\,\star_{\!\!\! M} \, \gc^{GD}$.
\end{theorem}

\begin{proof}
Firstly, we calculate
$$\gc\,\star_{\!\!\! M} \, \gc^{GD}\,\star_{\!\!\! M} \, \gc=\gc,\ \gc^{GD}\,\star_{\!\!\! M} \, \gc^{k+1}=\gc^k,\ \gc^{k+1}\,\star_{\!\!\! M} \, \gc^{GD}=\gc^k$$ and $$\hc\,\star_{\!\!\! M} \, \hc^{GD}\,\star_{\!\!\! M} \, \hc=\hc,\ \hc^{GD}\,\star_{\!\!\! M} \, \hc^{k+1}=\hc^k,\ \hc^{k+1}\,\star_{\!\!\! M} \, \hc^{GD}=\hc^k.$$
Now it is
\begin{align}\label{sss}
  \gc\,\star_{\!\!\! M} \, \hc\,\star_{\!\!\! M} \, \hc^{GD}\,\star_{\!\!\! M} \, \gc^{GD}\,\star_{\!\!\! M} \, \gc\,\star_{\!\!\! M} \, \hc&=
  \gc\,\star_{\!\!\! M} \, \gc^{GD}\,\star_{\!\!\! M} \, \hc\,\star_{\!\!\! M} \, \hc^{GD}\,\star_{\!\!\! M} \, \gc\,\star_{\!\!\! M} \, \hc
  =\gc\,\star_{\!\!\! M} \, \gc^{GD}\,\star_{\!\!\! M} \, \hc\,\star_{\!\!\! M} \, \hc^{GD}\,\star_{\!\!\! M} \, \hc\,\star_{\!\!\! M} \, \gc \nonumber \\
  &=\gc\,\star_{\!\!\! M} \, \gc^{GD}\,\star_{\!\!\! M} \, \hc\,\star_{\!\!\! M} \, \gc
  =\gc\,\star_{\!\!\! M} \, \gc^{GD}\,\star_{\!\!\! M} \, \gc\,\star_{\!\!\! M} \, \hc
  =\gc\,\star_{\!\!\! M} \, \hc,
\end{align}
\begin{align}\label{ddd}
 \hc^{GD}\,\star_{\!\!\! M} \, \gc^{GD}\,\star_{\!\!\! M} \, (\gc\,\star_{\!\!\! M} \, \hc)^{k+1}&= \hc^{GD}\,\star_{\!\!\! M} \, \gc^{GD}\,\star_{\!\!\! M} \, \gc^{k+1}\,\star_{\!\!\! M} \, \hc^{k+1}
  = \hc^{GD}\,\star_{\!\!\! M} \, \gc^{k}\,\star_{\!\!\! M} \, \hc^{k+1}\nonumber\\
  &= \hc^{GD}\,\star_{\!\!\! M} \, \hc^{k+1}\,\star_{\!\!\! M} \, \gc^{k}
  = \hc^k\,\star_{\!\!\! M} \, \gc^k
  =\gc^k\,\star_{\!\!\! M} \, \hc^k
  =(\gc\,\star_{\!\!\! M} \, \hc)^k,
\end{align}
and
\begin{align}\label{vvv}
 (\gc\,\star_{\!\!\! M} \, \hc)^{k+1}\,\star_{\!\!\! M} \, \hc^{GD}\,\star_{\!\!\! M} \, \gc^{GD}&=\gc^{k+1}\,\star_{\!\!\! M} \,
 \hc^{k+1}\,\star_{\!\!\! M} \, \hc^{GD}\,\star_{\!\!\! M} \, \gc^{GD}
 =\gc^{k+1}\,\star_{\!\!\! M} \, \hc^{k}\,\star_{\!\!\! M} \, \gc^{GD}\nonumber\\
 &=\hc^{k}\,\star_{\!\!\! M} \, \gc^{k+1}\,\star_{\!\!\! M} \, \gc^{GD}
 =\hc^{k}\,\star_{\!\!\! M} \, \gc^{k}
 =(\gc\,\star_{\!\!\! M} \, \hc)^k.
\end{align}
An application of $({\ref{sss}})$, $({\ref{ddd}})$, and $({\ref{vvv}})$ leads to $(\gc\,\star_{\!\!\! M} \, \hc)^{GD}=\hc^{GD}\,\star_{\!\!\! M} \, \gc^{GD}$.
\end{proof}

\begin{theorem}\label{tt7}
Let $\gc, \hc\in\mathbb{C}^{\eta_1\times \eta_1\times \eta_3}$ and $k=\max\{\ind(\gc), \ind(\hc)\}$. If
$$\gc\,\star_{\!\!\! M} \, \hc=\hc\,\star_{\!\!\! M} \, \gc, \ \ \ \hc^{GD}\,\star_{\!\!\! M} \, \hc\,\star_{\!\!\! M} \, \gc=\gc\,\star_{\!\!\! M} \, \hc^{GD}\,\star_{\!\!\! M} \, \hc,$$
then $(\gc\,\star_{\!\!\! M} \, \hc)^{GD}=\gc^{GD}\,\star_{\!\!\! M} \, \hc^{GD}$.
\end{theorem}

\begin{proof}
Clearly, we have $$\gc\,\star_{\!\!\! M} \, \gc^{GD}\,\star_{\!\!\! M} \, \gc=\gc,\ \gc^{GD}\,\star_{\!\!\! M} \, \gc^{k+1}=\gc^k,\  \gc^{k+1}\,\star_{\!\!\! M} \, \gc^{GD}=\gc^k$$
and
$$\hc\,\star_{\!\!\! M} \, \hc^{GD}\,\star_{\!\!\! M} \, \hc=\hc,\ \hc^{GD}\,\star_{\!\!\! M} \, \hc^{k+1}=\hc^k,\ \hc^{k+1}\,\star_{\!\!\! M} \, \hc^{GD}=\hc^k.$$ Now,
\begin{align}\label{eee}
  \gc\,\star_{\!\!\! M} \, \hc\,\star_{\!\!\! M} \, \gc^{GD}\,\star_{\!\!\! M} \, \hc^{GD}\,\star_{\!\!\! M} \, \gc\,\star_{\!\!\! M} \, \hc&=
  \hc\,\star_{\!\!\! M} \, \gc\,\star_{\!\!\! M} \, \gc^{GD}\,\star_{\!\!\! M} \, \hc^{GD}\,\star_{\!\!\! M} \, \hc\,\star_{\!\!\! M} \, \gc
  =\hc\,\star_{\!\!\! M} \, \gc\,\star_{\!\!\! M} \, \gc^{GD}\,\star_{\!\!\! M} \, \gc\,\star_{\!\!\! M} \, \hc^{GD}\,\star_{\!\!\! M} \, \hc \nonumber \\
  &=\hc\,\star_{\!\!\! M} \, \gc\,\star_{\!\!\! M} \, \hc^{GD}\,\star_{\!\!\! M} \, \hc
  =\gc\,\star_{\!\!\! M} \, \hc\,\star_{\!\!\! M} \, \hc^{GD}\,\star_{\!\!\! M} \, \hc
  =\gc\,\star_{\!\!\! M} \, \hc,
\end{align}
\begin{equation}\label{ppp}
\aligned
 \gc^{GD}\,\star_{\!\!\! M} \, \hc^{GD}\,\star_{\!\!\! M} \, (\gc\,\star_{\!\!\! M} \, \hc)^{k+1}&= \gc^{GD}\,\star_{\!\!\! M} \, \hc^{GD}\,\star_{\!\!\! M} \, \hc^{k+1}\,\star_{\!\!\! M} \, \gc^{k+1} \\
 &  = \gc^{GD}\,\star_{\!\!\! M} \, \hc^{k}\,\star_{\!\!\! M} \, \gc^{k+1}\\
  &= \gc^{GD}\,\star_{\!\!\! M} \, \gc^{k+1}\,\star_{\!\!\! M} \, \hc^{k}   = \gc^k\,\star_{\!\!\! M} \, \hc^k\\
  &=(\gc\,\star_{\!\!\! M} \, \hc)^k,
\endaligned
\end{equation}
and
\begin{align}\label{lll}
 (\gc\,\star_{\!\!\! M} \, \hc)^{k+1}\,\star_{\!\!\! M} \, \gc^{GD}\,\star_{\!\!\! M} \, \hc^{GD}&=\hc^{k+1}\,\star_{\!\!\! M} \,
 \gc^{k+1}\,\star_{\!\!\! M} \, \gc^{GD}\,\star_{\!\!\! M} \, \hc^{GD}  =\hc^{k+1}\,\star_{\!\!\! M} \, \gc^{k}\,\star_{\!\!\! M} \, \hc^{GD}\nonumber\\
 &=\gc^{k}\,\star_{\!\!\! M} \, \hc^{k+1}\,\star_{\!\!\! M} \, \hc^{GD}
 =\gc^{k}\,\star_{\!\!\! M} \, \hc^{k}
 =(\gc\,\star_{\!\!\! M} \, \hc)^k.
\end{align}
From $({\ref{eee}})$, $({\ref{ppp}})$, and $({\ref{lll}})$, we get $(\gc\,\star_{\!\!\! M} \, \hc)^{GD}=\gc^{GD}\,\star_{\!\!\! M} \, \hc^{GD}$.
\end{proof}

\begin{theorem}\label{tt8}
Let $\gc, \hc\in\mathbb{C}^{\eta_1\times \eta_1\times \eta_3}$ and $\gc\,\star_{\!\!\! M} \, \hc=\hc\,\star_{\!\!\! M} \, \gc=\mathcal{O}$. Suppose $k=\max\{\ind(\gc), \ind(\hc)\}$. If
$$\gc^{GD}\,\star_{\!\!\! M} \, \hc=\hc\,\star_{\!\!\! M} \, \gc^{GD}=\mathcal{O}, \  \ \  \hc^{GD}\,\star_{\!\!\! M} \, \gc=\gc\,\star_{\!\!\! M} \, \hc^{GD}=\mathcal{O},$$
then $(\gc+\hc)^{GD}=\gc^{GD}+\hc^{GD}$.
\end{theorem}

\begin{proof}
By $\gc\,\star_{\!\!\! M} \, \hc=\hc\,\star_{\!\!\! M} \, \gc=\mathcal {O}$, one has $(\gc+\hc)^n=\gc^{n}+\hc^{n}$ for arbitrary integer $n>0$.
Furthermore,
\begin{align}\label{iii}
  (\gc+\hc)\,\star_{\!\!\! M} \, (\gc^{GD}+\hc^{GD})\,\star_{\!\!\! M} \, (\gc+\hc)&=(\gc\,\star_{\!\!\! M} \, \gc^{GD}+
  \gc\,\star_{\!\!\! M} \, \hc^{GD}+\hc\,\star_{\!\!\! M} \, \gc^{GD}+\hc\,\star_{\!\!\! M} \, \hc^{GD})
  (\gc+\hc)\nonumber\\
  &=\gc\,\star_{\!\!\! M} \, \gc^{GD}\,\star_{\!\!\! M} \, \gc+\gc\,\star_{\!\!\! M} \, \hc^{GD}\,\star_{\!\!\! M} \, \gc+\hc\,\star_{\!\!\! M} \, \gc^{GD}\,\star_{\!\!\! M} \, \gc+\hc\,\star_{\!\!\! M} \, \hc^{GD}\,\star_{\!\!\! M} \, \gc\nonumber\\
  &+\gc\,\star_{\!\!\! M} \, \gc^{GD}\,\star_{\!\!\! M} \, \hc+\gc\,\star_{\!\!\! M} \, \hc^{GD}\,\star_{\!\!\! M} \, \hc+\hc\,\star_{\!\!\! M} \, \gc^{GD}\,\star_{\!\!\! M} \, \hc+\hc\,\star_{\!\!\! M} \, \hc^{GD}\,\star_{\!\!\! M} \, \hc\nonumber\\
  &=\gc\,\star_{\!\!\! M} \, \gc^{GD}\,\star_{\!\!\! M} \, \gc+\hc\,\star_{\!\!\! M} \, \hc^{GD}\,\star_{\!\!\! M} \, \hc\nonumber\\
  &=\gc+\hc,
\end{align}
and
\begin{align}\label{ooo}
 (\gc^{GD}+\hc^{GD})\,\star_{\!\!\! M} \, (\gc+\hc)^{k+1}&=(\gc^{GD}+\hc^{GD})\,\star_{\!\!\! M} \, (\gc^{k+1}+\hc^{k+1})\nonumber\\
 &=\gc^{GD}\,\star_{\!\!\! M} \, \gc^{k+1}+\gc^{GD}\,\star_{\!\!\! M} \, \hc^{k+1}+\hc^{GD}\,\star_{\!\!\! M} \, \gc^{k+1}+\hc^{GD}\,\star_{\!\!\! M} \, \hc^{k+1}\nonumber\\
 &=\gc^k+\gc^{GD}\,\star_{\!\!\! M} \, \hc\,\star_{\!\!\! M} \, \hc^k+\hc^{GD}\,\star_{\!\!\! M} \, \gc\,\star_{\!\!\! M} \, \gc^k+\hc^k\nonumber\\
 &=\gc^k+\hc^k\nonumber\\
 &=(\gc+\hc)^k,
\end{align}
 and
\begin{align}\label{uuu}
  (\gc+\hc)^{k+1}\,\star_{\!\!\! M} \, (\gc^{GD}+\hc^{GD})&=(\gc^{k+1}+\hc^{k+1})\,\star_{\!\!\! M} \, (\gc^{GD}+\hc^{GD})
 \nonumber\\
  &=\gc^{k+1}\,\star_{\!\!\! M} \, \gc^{GD}+\gc^{k+1}\,\star_{\!\!\! M} \, \hc^{GD}+\hc^{k+1}\,\star_{\!\!\! M} \, \gc^{GD}+\hc^{k+1}\,\star_{\!\!\! M} \, \hc^{GD}\nonumber\\
 &=\gc^k+\hc^k\nonumber\\
 &=(\gc+\hc)^k.
\end{align}
From $({\ref{iii}})$, $({\ref{ooo}})$, and $({\ref{uuu}})$, we get $(\gc+\hc)^{GD}=\gc^{GD}+\hc^{GD}$.
\end{proof}

\subsection{The GDMP inverse of a tensor}

The tensor GDMP inverse will be subject of investigation in this subsection. Before that, it is necessary to define the MP inverse in a tensor space.

\begin{definition}
Let $\gc \in \CC^{\eta_1\times \eta_2\times \eta_3}$ be fixed.
If there exists $\xc \in \CC^{\eta_2\times \eta_1\times \eta_3}$ satisfying
\begin{equation}\label{mpkk0}
\textnormal{(I)} \ \gc\,\star_{\!\!\! M} \, \xc\,\star_{\!\!\! M} \, \gc = \gc, \quad \textnormal{(II)} \ \xc\,\star_{\!\!\! M} \, \gc\,\star_{\!\!\! M} \, \xc = \xc, \quad \textnormal{(III)} \ (\gc\,\star_{\!\!\! M} \, \xc)^* = \gc\,\star_{\!\!\! M} \, \xc, \quad \textnormal{(IV)} \ (\xc\,\star_{\!\!\! M} \, \gc)^* = \xc\,\star_{\!\!\! M} \, \gc,
\end{equation}
then $\xc$ represents the {\em Moore-Penrose (MP) inverse} of $\gc$ and it is denoted by $\gc^\dag$.
\end{definition}

\begin{definition}\label{dd3.061}
Let $\gc \in\mathbb{C}^{\eta_1\times \eta_1\times \eta_3}$ and $\ind(\gc)=k$. Suppose $\gc^{GD}\in\gc\{GD\}$. Then, $\gc^{GD,\dag}=\gc^{GD}\,\star_{\!\!\! M} \, \gc\,\star_{\!\!\! M} \, \gc^\dag$ is call a {\em GDMP inverse} of $\gc$.
\end{definition}

Notice that the GD inverse is not unique and so is the GDMP inverse. We use $\gc\{GD, \dag\}=\{\gc^{GD, \dag}|\gc \in\mathbb{C}^{\eta_1\times \eta_1\times \eta_3}\}$ to denote the set of tensor GDMP inverses.

Next, an algorithm for computing the GDMP inverse is established.

~

\begin{algorithm}[H]
  \caption{Computing the GDMP inverse under the M-product}	
\KwIn{$\gc\in\mathbb{C}^{\eta_1\times \eta_1\times \eta_3}$ and $M \in\mathbb{C}^{\eta_3\times \eta_3}$}
\KwOut {$\xc=\gc^{GD,\dag}$}

\begin{enumerate}\label{ag}

\item Compute $\widehat{\gc}=\gc\times_3M$
\item $k=\ind(\widehat{\gc})$
\item $\textbf{For}$ $i=1:\eta_3$ \textbf{do}
\item ~~~~~~$(\widehat{\pc})^{(i)}=(\widehat{\gc^{GD}})^{(i)}(\widehat{\gc})^{(i)}(\widehat{\gc^\dag})^{(i)}$
\item \textbf{End for}
\item Compute $\xc=\widehat{\pc}\times_3M^{-1}$
\item \textbf{Return} $\gc^{GD,\dag}=\xc$.
\end{enumerate}

\end{algorithm}

\begin{example}
Let $\gc\in\mathbb{C}^{3\times3\times3}$ and $M\in\mathbb{C}^{3\times3}$  with entries
$$(\mathcal A)^{(1)}=
\begin{bmatrix}
3&3&5\\
4&4&7\\
2&3&4
\end{bmatrix},~
(\mathcal A)^{(2)}=
\begin{bmatrix}
2&2&4\\
3&3&4\\
2&2&3
\end{bmatrix},~
(\mathcal A)^{(3)}=
\begin{bmatrix}
0&2&4\\
1&1&4\\
0&2&3
\end{bmatrix},~
M=
\begin{bmatrix}
1&0&-1\\
0&1&0\\
0&1&-1
\end{bmatrix}.
$$
We evaluate the  index of $\gc$ is $k=2$, since $\ind((\widehat{\gc})^{(1)})=\ind((\widehat{\gc})^{(2)})=\ind((\widehat{\gc})^{(3)})=2$.
By Algorithm \ref{ag}, we calculate $\xc=\gc^{GD}\,\star_{\!\!\! M} \, \gc\,\star_{\!\!\! M} \, \gc^\dag$, where
$$
(\mathcal X)^{(1)}=
\begin{bmatrix}
0.3811&0.2619&-0.524\\
-1.3824&-0.6899&2.8104\\
1.7866&-1.5005&0.5711
\end{bmatrix},$$
$$(\mathcal X)^{(2)}=
\begin{bmatrix}
-5.1428&3.7144&0.5712\\
4.0952&-2.8096&-0.3808\\
0.7619&-0.4762&-0.0476
\end{bmatrix},~
(\mathcal X)^{(3)}=
\begin{bmatrix}
-7.119&4.262&2.476\\
3.119&-3.262&1.524\\
-1.2143&0.0714&1.8572
\end{bmatrix}.
$$
$\Box$
\end{example}

\begin{theorem}\label{mmk}
If $\gc \in\mathbb{C}^{\eta_1\times \eta_1\times \eta_3}$ is of the index $\ind(\gc)=k$, then the system of tensor equations
\begin{equation}\label{mpkk12}
\textnormal{(I)} \ \xc\,\star_{\!\!\! M} \, \gc\,\star_{\!\!\! M} \, \xc = \xc, \quad \textnormal{(II)} \ \gc\,\star_{\!\!\! M} \, \xc=\gc\,\star_{\!\!\! M} \, \gc^\dag, \quad \textnormal{(III)} \ \xc\,\star_{\!\!\! M} \, \gc^{k}=\gc^{GD}\,\star_{\!\!\! M} \, \gc^{k},
\end{equation}
possesses the solution $\xc=\gc^{GD}\,\star_{\!\!\! M} \, \gc\,\star_{\!\!\! M} \, \gc^\dag$.
\end{theorem}

\begin{proof}
Suppose $\xc=\gc^{GD}\,\star_{\!\!\! M} \, \gc\,\star_{\!\!\! M} \, \gc^\dag$. Then
\begin{align*}
\xc\,\star_{\!\!\! M} \, \gc\,\star_{\!\!\! M} \, \xc&=\gc^{GD}\,\star_{\!\!\! M} \, \gc\,\star_{\!\!\! M} \, \gc^\dag\,\star_{\!\!\! M} \, \gc\,\star_{\!\!\! M} \, \gc^{GD}\,\star_{\!\!\! M} \, \gc\,\star_{\!\!\! M} \, \gc^\dag
=\gc^{GD}\,\star_{\!\!\! M} \, \gc\,\star_{\!\!\! M} \, \gc^{GD}\,\star_{\!\!\! M} \, \gc\,\star_{\!\!\! M} \, \gc^\dag\\
&=\gc^{GD}\,\star_{\!\!\! M} \, \gc\,\star_{\!\!\! M} \, \gc^\dag
=\xc,
\end{align*}
$$\gc\,\star_{\!\!\! M} \, \xc=\gc\,\star_{\!\!\! M} \, \gc^{GD}\,\star_{\!\!\! M} \, \gc\,\star_{\!\!\! M} \, \gc^\dag=\gc\,\star_{\!\!\! M} \, \gc^\dag,$$
and
\begin{align*}
  \xc\,\star_{\!\!\! M} \, \gc^k&=\gc^{GD}\,\star_{\!\!\! M} \, \gc\,\star_{\!\!\! M} \, \gc^\dag\,\star_{\!\!\! M} \, \gc^k
  =\gc^{GD}\,\star_{\!\!\! M} \, \gc\,\star_{\!\!\! M} \, \gc^\dag\,\star_{\!\!\! M} \, \gc\,\star_{\!\!\! M} \, \gc^{k-1}
=\gc^{GD}\,\star_{\!\!\! M} \, \gc\,\star_{\!\!\! M} \, \gc^{k-1}
=\gc^{GD}\,\star_{\!\!\! M} \, \gc^K.
\end{align*}
The proof is completed.
\end{proof}

\begin{theorem}
Let $\gc \in\mathbb{C}^{\eta_1\times \eta_1\times \eta_3}$ satisfy $\ind(\gc)=k$. In this case, the system
\begin{equation}\label{cca}
\textnormal{(I)} \ \gc\,\star_{\!\!\! M} \, \xc\,\star_{\!\!\! M} \, \gc = \gc, \quad  \textnormal{(II)} \ \xc\,\star_{\!\!\! M} \, \gc\,\star_{\!\!\! M} \, \xc = \xc, \quad \textnormal{(III)} \ (\gc\,\star_{\!\!\! M} \, \xc)^* =\gc\,\star_{\!\!\! M} \, \xc,  \quad  \textnormal{(IV)} \ \xc\,\star_{\!\!\! M} \, \gc^{k+1} =\gc^k,
\end{equation}
possesses the solution $\xc=\gc^{GD}\,\star_{\!\!\! M} \, \gc\,\star_{\!\!\! M} \, \gc^\dag$.
\end{theorem}

\begin{proof}
Suppose $\xc=\gc^{GD}\,\star_{\!\!\! M} \, \gc\,\star_{\!\!\! M} \, \gc^\dag$.
By Theorem \ref{mmk}, it is enough to verify\\
$\textnormal{(I)} \ \gc\,\star_{\!\!\! M} \, \xc\,\star_{\!\!\! M} \, \gc = \gc,  \quad \textnormal{(III)} \ (\gc\,\star_{\!\!\! M} \, \xc)^* =\gc\,\star_{\!\!\! M} \, \xc,  \quad  \textnormal{(IV)} \ \xc\,\star_{\!\!\! M} \, \gc^{k+1} =\gc^k$. \\
The following identities are satisfied
$$\gc\,\star_{\!\!\! M} \, \gc^{GD}\,\star_{\!\!\! M} \, \gc\,\star_{\!\!\! M} \, \gc^\dag\,\star_{\!\!\! M} \, \gc=\gc,\ \ \
\gc\,\star_{\!\!\! M} \, \xc=\gc\,\star_{\!\!\! M} \, \gc^{GD}\,\star_{\!\!\! M} \, \gc\,\star_{\!\!\! M} \, \gc^\dag=\gc\,\star_{\!\!\! M} \, \gc^\dag=(\gc\,\star_{\!\!\! M} \, \xc)^*,$$
as well as
$$\xc\,\star_{\!\!\! M} \, \gc^{k+1}=\gc^{GD}\,\star_{\!\!\! M} \, \gc\,\star_{\!\!\! M} \, \gc^\dag\,\star_{\!\!\! M} \, \gc^{k+1}=
\gc^{GD}\,\star_{\!\!\! M} \, \gc\,\star_{\!\!\! M} \, \gc^k=\gc^{GD}\,\star_{\!\!\! M} \, \gc^{k+1}
=\gc^k.$$
\end{proof}

Some properties of the GDMP inverse are established in Theorem \ref{Thm11TM}.
\begin{theorem}\label{Thm11TM}
Let $\gc \in\mathbb{C}^{\eta_1\times \eta_1\times \eta_3}$ fulfill $\ind(\gc)=k$ and assume $\gc^{GD}\in\gc\{GD\}$. Then,
\begin{description}
  \item  [(a)] $\gc^{GD,\dag}=\gc^{GD}\,\star_{\!\!\! M} \, \gc\,\star_{\!\!\! M} \, \gc^{GD,\dag}$.
  \item  [(b)] $\gc^{GD,\dag}\in\gc\{1,2\}.$
  %\item  [(c)] $\gc\,\star_{\!\!\! M} \, \gc^{GD,\dag}=\gc\,\star_{\!\!\! M} \, \gc^\dag$ and $\gc^{GD,\dag}\,\star_{\!\!\! M} \, \gc=\gc^{GD}\,\star_{\!\!\! M} \, \gc$.
  \item  [(c)] $\gc^c\,\star_{\!\!\! M} \, \gc^{GD,\dag}=\gc^c\,\star_{\!\!\! M} \, \gc^\dag$ and $\gc^{GD,\dag}\,\star_{\!\!\! M} \, \gc^c=\gc^{GD}\,\star_{\!\!\! M} \, \gc^c$, where $c$ is a positive integer.
  \item  [(d)] $\gc^{GD,\dag}\,\star_{\!\!\! M} \, \gc^{k+1}=\gc^k$.
  \item  [(e)] $\gc\,\star_{\!\!\! M} \, \gc^{GD,\dag}\,\star_{\!\!\! M} \, \gc^{k+1}=\gc^{k+1}$.
  \item  [(f)] $\gc^{k+1}\,\star_{\!\!\! M} \, \gc^{GD,\dag}=\gc^{k}\,\star_{\!\!\! M} \, \gc\,\star_{\!\!\! M} \, \gc^\dag$.
  \item  [(g)] $\gc^{GD,\dag}\,\star_{\!\!\! M} \, \gc^{k+2}\,\star_{\!\!\! M} \, \gc^{GD,\dag}=\gc^{k+1}\,\star_{\!\!\! M} \, \gc^\dag$.
\end{description}
\end{theorem}

\begin{proof}  \textbf{(a)}
$\gc^{GD,\dag}=\gc^{GD,\dag}\,\star_{\!\!\! M} \, \gc\,\star_{\!\!\! M} \, \gc^{GD,\dag}
    =\gc^{GD}\,\star_{\!\!\! M} \, \gc\,\star_{\!\!\! M} \, \gc^\dag\,\star_{\!\!\! M} \, \gc\,\star_{\!\!\! M} \, \gc^{GD,\dag}=\gc^{GD}\,\star_{\!\!\! M} \, \gc\,\star_{\!\!\! M} \, \gc^{GD,\dag}$.\\
\textbf{(b)} Based on Theorem \ref{mmk}, it suffices to prove $\gc\,\star_{\!\!\! M} \, \gc^{GD,\dag}\,\star_{\!\!\! M} \, \gc=\gc$.
Since
$$\gc\,\star_{\!\!\! M} \, \gc^{GD,\dag}\,\star_{\!\!\! M} \, \gc=\gc\,\star_{\!\!\! M} \, \gc^{GD}\,\star_{\!\!\! M} \, \gc\,\star_{\!\!\! M} \, \gc^\dag\,\star_{\!\!\! M} \, \gc=\gc,$$
we have $\gc^{GD,\dag}\in\gc\{1,2\}$.\\
\textbf{(c)}
Let $c$ is a positive integer. Then,
\begin{eqnarray*}
% \nonumber to remove numbering (before each equation)
  \gc^c\,\star_{\!\!\! M} \, \gc^{GD,\dag} &=& \gc^c\,\star_{\!\!\! M} \, \gc^{GD}\,\star_{\!\!\! M} \, \gc\,\star_{\!\!\! M} \, \gc^\dag
=\gc^{c-1}\,\star_{\!\!\! M} \, (\gc\,\star_{\!\!\! M} \, \gc^{GD}\,\star_{\!\!\! M} \, \gc)\,\star_{\!\!\! M} \, \gc^\dag \\
   &=& \gc^{c-1}\,\star_{\!\!\! M} \, \gc\,\star_{\!\!\! M} \, \gc^\dag=\gc^{c}\,\star_{\!\!\! M} \, \gc^\dag.
\end{eqnarray*}
Similarly, we can prove $\gc^{GD,\dag}\,\star_{\!\!\! M} \, \gc^c=\gc^{GD}\,\star_{\!\!\! M} \, \gc^c$.

\noindent \textbf{(d)} Using $\gc^{GD,\dag}\,\star_{\!\!\! M} \, \gc^c=\gc^{GD}\,\star_{\!\!\! M} \, \gc^c$, it can be concluded
$$\gc^{GD,\dag}\,\star_{\!\!\! M} \, \gc^{k+1}=\gc^{GD,\dag}\,\star_{\!\!\! M} \, \gc^k\,\star_{\!\!\! M} \, \gc=\gc^{GD}\,\star_{\!\!\! M} \, \gc^k\,\star_{\!\!\! M} \, \gc=\gc^{GD}\,\star_{\!\!\! M} \, \gc^{k+1}=\gc^k.$$
\textbf{(e)} Since $\gc^{GD,\dag}\,\star_{\!\!\! M} \, \gc=\gc^{GD}\,\star_{\!\!\! M} \, \gc$ and $\gc^{GD}\,\star_{\!\!\! M} \, \gc^{k+1}=\gc^k$, one has
\begin{eqnarray*}
% \nonumber to remove numbering (before each equation)
  \gc\,\star_{\!\!\! M} \, \gc^{GD,\dag}\,\star_{\!\!\! M} \, \gc^{k+1} &=& \gc\,\star_{\!\!\! M} \, \gc^{GD,\dag}\,\star_{\!\!\! M} \, \gc\,\star_{\!\!\! M} \, \gc^k=
  \gc\,\star_{\!\!\! M} \, \gc^{GD}\,\star_{\!\!\! M} \, \gc\,\star_{\!\!\! M} \, \gc^k \\
   &=& \gc\,\star_{\!\!\! M} \, \gc^{GD}\,\star_{\!\!\! M} \, \gc^{k+1}=\gc\,\star_{\!\!\! M} \, \gc^k=\gc^{k+1}.
\end{eqnarray*}
\textbf{(f)} It is easy to see
$$\gc^{k+1}\,\star_{\!\!\! M} \, \gc^{GD,\dag}=\gc^k\,\star_{\!\!\! M} \, \gc\,\star_{\!\!\! M} \, \gc^{GD,\dag}=\gc^k\,\star_{\!\!\! M} \, (\gc\,\star_{\!\!\! M} \, \gc^{GD,\dag})=\gc^k\,\star_{\!\!\! M} \, \gc\,\star_{\!\!\! M} \, \gc^\dag.$$
\textbf{(g)} By some computation, it can be derived
$$\gc^{GD,\dag}\,\star_{\!\!\! M} \, \gc^{k+2}\,\star_{\!\!\! M} \, \gc^{GD,\dag}=\gc^{GD,\dag}\,\star_{\!\!\! M} \, \gc^{k+1}\,\star_{\!\!\! M} \, \gc\,\star_{\!\!\! M} \, \gc^{GD,\dag}
=\gc^k\,\star_{\!\!\! M} \, \gc\,\star_{\!\!\! M} \, \gc^\dag=\gc^{k+1}\,\star_{\!\!\! M} \, \gc^\dag,$$
which finishes the proof.
\end{proof}

The subsequent results are aimed to the reverse- and forward-order law for the GDMP inverse.

\begin{theorem}
Let $\gc, \hc\in\mathbb{C}^{\eta_1\times \eta_1\times \eta_3}$ satisfy the constraint $\gc^\dag=\gc$, $\hc^\dag=\hc$.
Let us observe $k=\max\{\ind(\gc), \ind(\hc)\}$.
If
$$\gc\,\star_{\!\!\! M} \, \hc=\hc\,\star_{\!\!\! M} \, \gc, \ \ \ \gc^{GD}\,\star_{\!\!\! M} \, \gc\,\star_{\!\!\! M} \, \hc=\hc\,\star_{\!\!\! M} \, \gc^{GD}\,\star_{\!\!\! M} \, \gc,\ \ \
 \hc\,\star_{\!\!\! M} \, \hc^{GD}\,\star_{\!\!\! M} \, \gc^{GD}=\gc^{GD}\,\star_{\!\!\! M} \, \hc\,\star_{\!\!\! M} \, \hc^{GD},$$
then $(\gc\,\star_{\!\!\! M} \, \hc)^{GD,\dag}=\hc^{GD,\dag}\,\star_{\!\!\! M} \, \gc^{GD,\dag}$.
\end{theorem}

\begin{proof}
Since $\gc^\dag=\gc$ and $\hc^\dag=\hc$ and $\gc\,\star_{\!\!\! M} \, \hc=\hc\,\star_{\!\!\! M} \, \gc$, it follows
$$(\gc\,\star_{\!\!\! M} \, \hc)^\dag=\hc^\dag\,\star_{\!\!\! M} \, \gc^\dag=\gc^\dag\,\star_{\!\!\! M} \, \hc^\dag.$$
Now, from Theorem {\ref{T5}}, we conclude
\begin{align*}
  (\gc\,\star_{\!\!\! M} \, \hc)^{GD,\dag}&=(\gc\,\star_{\!\!\! M} \, \hc)^{GD}\,\star_{\!\!\! M} \, \gc\,\star_{\!\!\! M} \, \hc\,\star_{\!\!\! M} \, (\gc\,\star_{\!\!\! M} \, \hc)^\dag
  =\hc^{GD}\,\star_{\!\!\! M} \, \gc^{GD}\,\star_{\!\!\! M} \, \gc\,\star_{\!\!\! M} \, \hc\,\star_{\!\!\! M} \, \hc^\dag\,\star_{\!\!\! M} \, \gc^\dag\\
  &=\hc^{GD}\,\star_{\!\!\! M} \, \gc^{GD}\,\star_{\!\!\! M} \, \gc\,\star_{\!\!\! M} \, \hc\,\star_{\!\!\! M} \, \hc\,\star_{\!\!\! M} \, \gc
  =\hc^{GD}\,\star_{\!\!\! M} \, \hc\,\star_{\!\!\! M} \, \gc^{GD}\,\star_{\!\!\! M} \, \gc\,\star_{\!\!\! M} \, \hc\,\star_{\!\!\! M} \, \gc\\
  &=\hc^{GD}\,\star_{\!\!\! M} \, \hc\,\star_{\!\!\! M} \, \hc\,\star_{\!\!\! M} \, \gc^{GD}\,\star_{\!\!\! M} \, \gc\,\star_{\!\!\! M} \, \gc
  =\hc^{GD}\,\star_{\!\!\! M} \, \hc\,\star_{\!\!\! M} \, \hc^\dag\,\star_{\!\!\! M} \, \gc^{GD}\,\star_{\!\!\! M} \, \gc\,\star_{\!\!\! M} \, \gc^\dag\\
  &=\hc^{GD,\dag}\gc^{GD,\dag}.
\end{align*}
Hence, the desired reverse order property is verified.
\end{proof}

\begin{theorem}
Let $\gc, \hc\in\mathbb{C}^{\eta_1\times \eta_1\times \eta_3}$ and $\gc^\dag=\gc$, $\hc^\dag=\hc$. Suppose  $k=\max\{\ind(\gc), \ind(\hc)\}$. If
$$\gc\,\star_{\!\!\! M} \, \hc=\hc\,\star_{\!\!\! M} \, \gc, \ \ \ \hc^{GD}\,\star_{\!\!\! M} \, \hc\,\star_{\!\!\! M} \, \gc=\gc\,\star_{\!\!\! M} \, \hc^{GD}\,\star_{\!\!\! M} \, \hc,$$
then
$(\gc\,\star_{\!\!\! M} \, \hc)^{GD,\dag}=\gc^{GD,\dag}\,\star_{\!\!\! M} \, \hc^{GD,\dag}$.
\end{theorem}

\begin{proof}
Since $\gc^\dag=\gc$ and $\hc^\dag=\hc$ and $\gc\,\star_{\!\!\! M} \, \hc=\hc\,\star_{\!\!\! M} \, \gc$, we get
$$(\gc\,\star_{\!\!\! M} \, \hc)^\dag=\hc^\dag\,\star_{\!\!\! M} \, \gc^\dag=\gc^\dag\,\star_{\!\!\! M} \, \hc^\dag.$$
Now, results of Theorem {\ref{tt7}} imply
\begin{align*}
  (\gc\,\star_{\!\!\! M} \, \hc)^{GD,\dag}&=(\gc\,\star_{\!\!\! M} \, \hc)^{GD}\,\star_{\!\!\! M} \, \gc\,\star_{\!\!\! M} \, \hc\,\star_{\!\!\! M} \, (\gc\,\star_{\!\!\! M} \, \hc)^\dag
  =\gc^{GD}\,\star_{\!\!\! M} \, \hc^{GD}\,\star_{\!\!\! M} \, \gc\,\star_{\!\!\! M} \, \hc\,\star_{\!\!\! M} \, \hc^\dag\,\star_{\!\!\! M} \, \gc^\dag\\
  &=\gc^{GD}\,\star_{\!\!\! M} \, \hc^{GD}\,\star_{\!\!\! M} \, \gc\,\star_{\!\!\! M} \, \hc\,\star_{\!\!\! M} \, \hc\,\star_{\!\!\! M} \, \gc=
  \gc^{GD}\,\star_{\!\!\! M} \, \hc^{GD}\,\star_{\!\!\! M} \, \hc\,\star_{\!\!\! M} \, \gc\,\star_{\!\!\! M} \, \hc\,\star_{\!\!\! M} \, \gc\\
  &=\gc^{GD}\,\star_{\!\!\! M} \, \gc\,\star_{\!\!\! M} \, \hc^{GD}\,\star_{\!\!\! M} \, \hc\,\star_{\!\!\! M} \, \hc\,\star_{\!\!\! M} \, \gc
  =\gc^{GD}\,\star_{\!\!\! M} \, \gc\,\star_{\!\!\! M} \, \hc^{GD}\,\star_{\!\!\! M} \, \hc\,\star_{\!\!\! M} \, \gc\,\star_{\!\!\! M} \, \hc\\
  &=\gc^{GD}\,\star_{\!\!\! M} \, \gc\,\star_{\!\!\! M} \, \gc\,\star_{\!\!\! M} \, \hc^{GD}\,\star_{\!\!\! M} \, \hc\,\star_{\!\!\! M} \, \hc
  =\gc^{GD}\,\star_{\!\!\! M} \, \gc\,\star_{\!\!\! M} \, \gc^\dag\,\star_{\!\!\! M} \, \hc^{GD}\,\star_{\!\!\! M} \, \hc\,\star_{\!\!\! M} \, \hc^\dag\\
  &=\gc^{GD,\dag}\,\star_{\!\!\! M} \, \hc^{GD,\dag}.
\end{align*}
Hence, the statement is verified.
\end{proof}

\begin{theorem}
Let us suppose the constraints $\gc^\dag=\gc$, $\hc^\dag=\hc$, $\gc\,\star_{\!\!\! M} \, \hc=\hc\,\star_{\!\!\! M} \, \gc=\mathcal{O}$ for  $\gc, \hc\in\mathbb{C}^{\eta_1\times \eta_1\times \eta_3}$ and observe $k=\max\{\ind(\gc), \ind(\hc)\}$.  If
$$\gc^{GD}\,\star_{\!\!\! M} \, \hc=\hc\,\star_{\!\!\! M} \, \gc^{GD}=\mathcal{O}, \ \ \  \hc^{GD}\,\star_{\!\!\! M} \, \gc=\gc\,\star_{\!\!\! M} \, \hc^{GD}=\mathcal{O},$$
then $(\gc+\hc)^{GD,\dag}=\gc^{GD,\dag}+\hc^{GD,\dag}$.
\end{theorem}

\begin{proof}
Since $\gc^\dag=\gc$, $\hc^\dag=\hc$ and $\gc\,\star_{\!\!\! M} \, \hc=\hc\,\star_{\!\!\! M} \, \gc=\mathcal{O}$, we have $(\gc+\hc)^\dag=\gc^\dag+\hc^\dag$. From Theorem \ref{tt8}, we have $(\gc+\hc)^{GD}=\gc^{GD}+\hc^{GD}$. Now,
\begin{align*}
  (\gc+\hc)^{GD,\dag}&=(\gc+\hc)^{GD}\,\star_{\!\!\! M} \, (\gc+\hc)\,\star_{\!\!\! M} \, (\gc+\hc)^\dag
  =(\gc^{GD}+\hc^{GD})\,\star_{\!\!\! M} \, (\gc+\hc)\,\star_{\!\!\! M} \, (\gc^\dag+\hc^\dag)\\
  &=(\gc^{GD}\,\star_{\!\!\! M} \, \gc+\gc^{GD}\,\star_{\!\!\! M} \, \hc+\hc^{GD}\,\star_{\!\!\! M} \, \gc+\hc^{GD}\,\star_{\!\!\! M} \, \hc)\,\star_{\!\!\! M} \, (\gc^\dag+\hc^\dag)\\
  &=(\gc^{GD}\,\star_{\!\!\! M} \, \gc+\hc^{GD}\,\star_{\!\!\! M} \, \hc)\,\star_{\!\!\! M} \, (\gc^\dag+\hc^\dag)\\
  &=\gc^{GD}\,\star_{\!\!\! M} \, \gc\,\star_{\!\!\! M} \, \gc^\dag+
  \gc^{GD}\,\star_{\!\!\! M} \, \gc\,\star_{\!\!\! M} \, \hc^\dag+\hc^{GD}\,\star_{\!\!\! M} \, \hc\,\star_{\!\!\! M} \, \gc^\dag+\hc^{GD}\,\star_{\!\!\! M} \, \hc\,\star_{\!\!\! M} \, \hc^\dag\\
  %&=\gc^{GD,\dag}+\gc^{GD}\,\star_{\!\!\! M} \, \gc\,\star_{\!\!\! M} \, \hc+\hc^{GD}\,\star_{\!\!\! M} \, \hc\,\star_{\!\!\! M} \, \gc+\hc^{GD,\dag}\\
  &=\gc^{GD,\dag}+\hc^{GD,\dag}.
\end{align*}
Hence,
$(\gc\,\star_{\!\!\! M} \, \hc)^{GD,\dag}=
\gc^{GD,\dag}\,\star_{\!\!\! M} \, \hc^{GD,\dag}$.
\end{proof}

\subsection{The GD-Star inverse of a tensor}

In this subsection, we will establish and examine the GD-Star inverse of a tensor. Firstly, we give the definition of the  GD-Star inverse.
\begin{definition}\label{dd3.061}
Let $\gc \in\mathbb{C}^{\eta_1\times \eta_1\times \eta_3}$, $\ind(\gc)=k$., and $\gc^{GD}\in\gc\{GD\}$.
Then, $\gc^{GD,*}=\gc^{GD}\,\star_{\!\!\! M} \, \gc\,\star_{\!\!\! M} \, \gc^*$ is called a {\em GD-Star inverse} of $\gc$.

\end{definition}

Since the GD inverse is not unique and so is the GD-Star inverse.
The notation $\gc\{GD, *\}=\{\gc^{GD, *}|\gc \in\mathbb{C}^{\eta_1\times \eta_1\times \eta_3}\}$ will stand for the set of GD-Star inverses of $\gc$.

Algorithm \ref{Alg33TM} describes procedure for computing the GD-Star inverse.

\begin{algorithm}[H]\label{Alg33TM}
  \caption{Computing the GD-Star inverse under the M-product}	
\KwIn{$\gc\in\mathbb{C}^{\eta_1\times \eta_1\times \eta_3}$ and $M \in\mathbb{C}^{\eta_3\times \eta_3}$}
\KwOut {$\xc=\gc^{GD,*}$}

\begin{enumerate}\label{ao}

\item Compute $\widehat{\gc}=\gc\times_3M$
\item $k=\ind(\widehat{\gc})$
\item $\textbf{For}$ $i=1:\eta_3$ \textbf{do}
\item ~~~~~~$\widehat{\mathcal{R}}^{(i)}=(\widehat{\gc^{GD}})^{(i)}(\widehat{\gc})^{(i)}(\widehat{\gc^*})^{(i)}$
\item \textbf{End for}
\item Compute $\xc=\widehat{\mathcal{R}}\times_3M^{-1}$

\item \textbf{Return} $\gc^{GD,^*}=\xc$.
\end{enumerate}

\end{algorithm}

\begin{example}
Let $\gc\in\mathbb{C}^{3\times3\times3}$ and $M\in\mathbb{C}^{3\times3}$ with entries
$$(\mathcal A)^{(1)}=
\begin{bmatrix}
5&1&1\\
6&0&1\\
5&2&3
\end{bmatrix},~
(\mathcal A)^{(2)}=
\begin{bmatrix}
2&0&0\\
1&0&1\\
1&0&2
\end{bmatrix},~
(\mathcal A)^{(3)}=
\begin{bmatrix}
2&-1&-1\\
1&-2&-1\\
1&-1&0
\end{bmatrix},~
M=
\begin{bmatrix}
1&0&-1\\
0&1&0\\
0&1&-1
\end{bmatrix}.
$$
We evaluate the index of $\gc$ is $k=2$, since $\ind((\widehat{\gc})^{(1)})=\ind((\widehat{\gc})^{(2)})=\ind((\widehat{\gc})^{(3)})=2$.
By Algorithm \ref{ao}, we calculate $\xc=\gc^{GD}\,\star_{\!\!\! M} \, \gc\,\star_{\!\!\! M} \, \gc^*$, where
$$
(\mathcal X)^{(1)}=
\begin{bmatrix}
0.0066&-3.985&-9.9922\\
-12.1586&-27.565&-33.0738\\
-2.3238&-4.645&-14.6554
\end{bmatrix},$$
$$(\mathcal X)^{(2)}=
\begin{bmatrix}
2&1&1\\
0.5&0.75&1.25\\
0&1&2
\end{bmatrix},~
(\mathcal X)^{(3)}=
\begin{bmatrix}
-3&-9&-14\\
-10.5&-21.5&-31.75\\
-8&-15&-22
\end{bmatrix}.
$$
$\Box$
\end{example}

\begin{theorem}\label{T15}
Let $\gc \in\mathbb{C}^{\eta_1\times \eta_1\times \eta_3}$ is of the index $\ind(\gc)=k$.
For each $c\geq1$, the system
\begin{equation*}
 \textnormal{(I)} \ \xc\,\star_{\!\!\! M} \, (\gc^\dag)^*\,\star_{\!\!\! M} \, \xc=\xc, \ \ \ \textnormal{(II)} \   \gc^c\,\star_{\!\!\! M} \, \xc=\gc^c\,\star_{\!\!\! M} \, \gc^*, \ \ \
 \textnormal{(III)} \ \xc\,\star_{\!\!\! M} \, (\gc^\dag)^*=\gc^{GD}\,\star_{\!\!\! M} \, \gc
\end{equation*}
possesses the solution $\xc=\gc^{GD}\,\star_{\!\!\! M} \, \gc\,\star_{\!\!\! M} \, \gc^*$.
\end{theorem}

\begin{proof}
The proof is based on the verification that $\xc=\gc^{GD}\,\star_{\!\!\! M} \, \gc\,\star_{\!\!\! M} \, \gc^*$ satisfies the given tensor equations.
The tensor equation (I) is verified by the following identities:
\begin{align*}
 \xc\,\star_{\!\!\! M} \, (\gc^\dag)^*\,\star_{\!\!\! M} \, \xc  &=\gc^{GD}\,\star_{\!\!\! M} \, \gc\,\star_{\!\!\! M} \, \gc^*\,\star_{\!\!\! M} \, (\gc^\dag)^*\,\star_{\!\!\! M} \, \gc^{GD}\,\star_{\!\!\! M} \, \gc\,\star_{\!\!\! M} \, \gc^*\\
& =\gc^{GD}\,\star_{\!\!\! M} \, \gc\,\star_{\!\!\! M} \, (\gc^\dag\,\star_{\!\!\! M} \, \gc)^*\,\star_{\!\!\! M} \, \gc^{GD}\,\star_{\!\!\! M} \, \gc\,\star_{\!\!\! M} \, \gc^*\\
 &=\gc^{GD}\,\star_{\!\!\! M} \, \gc\,\star_{\!\!\! M} \, \gc^\dag\,\star_{\!\!\! M} \, \gc\,\star_{\!\!\! M} \, \gc^{GD}\,\star_{\!\!\! M} \, \gc\,\star_{\!\!\! M} \, \gc^*
 =\gc^{GD}\,\star_{\!\!\! M} \, \gc\,\star_{\!\!\! M} \, \gc^{GD}\,\star_{\!\!\! M} \, \gc\,\star_{\!\!\! M} \, \gc^*\\
 &=\gc^{GD}\,\star_{\!\!\! M} \, \gc\,\star_{\!\!\! M} \, \gc^*=\xc.
\end{align*}
Under the assumption $c\geq1$ it can be verified
\begin{align*}
  \gc^c\,\star_{\!\!\! M} \, \xc=\gc^c\,\star_{\!\!\! M} \, \gc^{GD}\,\star_{\!\!\! M} \, \gc\,\star_{\!\!\! M} \, \gc^*
  =\gc^{c-1}\,\star_{\!\!\! M} \, \gc\,\star_{\!\!\! M} \, \gc^{GD}\,\star_{\!\!\! M} \, \gc\,\star_{\!\!\! M} \, \gc^*
  =\gc^{c-1}\,\star_{\!\!\! M} \, \gc\,\star_{\!\!\! M} \, \gc^*=\gc^c\,\star_{\!\!\! M} \, \gc^*,
\end{align*}
and
\begin{align*}
  \xc\,\star_{\!\!\! M} \, (\gc^\dag)^*
  =\gc^{GD}\,\star_{\!\!\! M} \, \gc\,\star_{\!\!\! M} \, \gc^*\,\star_{\!\!\! M} \, (\gc^\dag)^*
  =\gc^{GD}\,\star_{\!\!\! M} \, \gc\,\star_{\!\!\! M} \, (\gc^\dag\,\star_{\!\!\! M} \, \gc)^*=\gc^{GD}\,\star_{\!\!\! M} \, \gc\,\star_{\!\!\! M} \, \gc^\dag\,\star_{\!\!\! M} \, \gc
  =\gc^{GD}\,\star_{\!\!\! M} \, \gc.
\end{align*}
Thus, $\gc^{GD,*}=\gc^{GD}\,\star_{\!\!\! M} \, \gc\,\star_{\!\!\! M} \, \gc^*$.
\end{proof}

\begin{theorem}\label{Thm16TM}
Let $\gc \in\mathbb{C}^{\eta_1\times \eta_1\times \eta_3}$ possesses the index $\ind(\gc)=k$. Then,
\begin{description}
  \item  [(a)] $\gc\,\star_{\!\!\! M} \, \gc^{GD,*}\,\star_{\!\!\! M} \, (\gc^\dag)^*=\gc$.
  \item  [(b)] $\gc^{\dag}\,\star_{\!\!\! M} \, \gc\,\star_{\!\!\! M} \, \gc^{GD,*}=\gc^*$.
  \item  [(c)] $\gc^k\,\star_{\!\!\! M} \, \gc^{GD,*}\,\star_{\!\!\! M} \, (\gc^\dag)^*=\gc^k$ and $\gc^{GD,*}\,\star_{\!\!\! M} \, (\gc^\dag)^*\,\star_{\!\!\! M} \, \gc^k=\gc^k$.
  \item  [(d)] $\gc^{\dag}\,\star_{\!\!\! M} \, \gc\,\star_{\!\!\! M} \, (\gc^{GD,*})^2=\gc^*\,\star_{\!\!\! M} \, \gc^{GD,*}$ and $\gc^{\dag}\,\star_{\!\!\! M} \, \gc\,\star_{\!\!\! M} \, (\gc^{GD,*})^2\,\star_{\!\!\! M} \, \gc\,\star_{\!\!\! M} \, \gc^{\dag}=\gc^*\,\star_{\!\!\! M} \, \gc^{GD,*}$.
  \item  [(e)] $(\gc\,\star_{\!\!\! M} \, \gc^{GD,*})^*=\gc\,\star_{\!\!\! M} \, \gc^{GD,*}$.
  \item  [(f)] $(\gc^\dag)^*\,\star_{\!\!\! M} \, \gc^{GD,*}\,\star_{\!\!\! M} \, (\gc^\dag)^*=(\gc^\dag)^*$.
\end{description}
\end{theorem}

\begin{proof}
\textbf{(a)} Some computation shows
$$\gc\,\star_{\!\!\! M} \, \gc^{GD,*}\,\star_{\!\!\! M} \, (\gc^\dag)^*    =\gc\,\star_{\!\!\! M} \, \gc^{GD}\,\star_{\!\!\! M} \, \gc\,\star_{\!\!\! M} \, \gc^*\,\star_{\!\!\! M} \, (\gc^\dag)^*
    =\gc\,\star_{\!\!\! M} \, (\gc^\dag\,\star_{\!\!\! M} \, \gc)^*=\gc\,\star_{\!\!\! M} \, \gc^\dag\,\star_{\!\!\! M} \, \gc=\gc.$$
\textbf{(b)} It is easy to see that
$$\gc^{\dag}\,\star_{\!\!\! M} \, \gc\,\star_{\!\!\! M} \, \gc^{GD,*}
    =\gc^\dag\,\star_{\!\!\! M} \, \gc\,\star_{\!\!\! M} \, \gc^{GD}\,\star_{\!\!\! M} \, \gc\,\star_{\!\!\! M} \, \gc^*
    =\gc^\dag\,\star_{\!\!\! M} \, \gc\,\star_{\!\!\! M} \, \gc^*
    =\gc^*\,\star_{\!\!\! M} \, (\gc^*)^\dag\,\star_{\!\!\! M} \, \gc^*=\gc^*.$$
\textbf{(c)} Clearly,
\begin{eqnarray*}
% \nonumber to remove numbering (before each equation)
  \gc^k\,\star_{\!\!\! M} \, \gc^{GD,*}\,\star_{\!\!\! M} \, (\gc^\dag)^*
    &=&\gc^k\,\star_{\!\!\! M} \, \gc^{GD}\,\star_{\!\!\! M} \, \gc\,\star_{\!\!\! M} \, (\gc^\dag\,\star_{\!\!\! M} \, \gc)^*
    =\gc^k\,\star_{\!\!\! M} \, \gc^{GD}\,\star_{\!\!\! M} \, \gc\,\star_{\!\!\! M} \, \gc^\dag\,\star_{\!\!\! M} \, \gc \\
   &=& \gc^k\,\star_{\!\!\! M} \, \gc^{GD}\,\star_{\!\!\! M} \, \gc=\gc^{k-1}\,\star_{\!\!\! M} \, \gc=\gc^k.
\end{eqnarray*}
Similarly, we can obtain $$\gc^{GD,*}\,\star_{\!\!\! M} \, (\gc^\dag)^*\,\star_{\!\!\! M} \, \gc^k
    =\gc^{GD}\,\star_{\!\!\! M} \, \gc\,\star_{\!\!\! M} \, \gc^*\,\star_{\!\!\! M} \, (\gc^\dag)^*\,\star_{\!\!\! M} \, \gc^k=\gc^{GD}\,\star_{\!\!\! M} \, \gc^{k+1}=\gc^k.$$
\textbf{(d)} Firstly, we verify the following identities:
\begin{eqnarray*}
% \nonumber to remove numbering (before each equation)
  \gc^{\dag}\,\star_{\!\!\! M} \, \gc\,\star_{\!\!\! M} \, (\gc^{GD,*})^2 &=& \gc^{\dag}\,\star_{\!\!\! M} \, \gc\,\star_{\!\!\! M} \, \gc^{GD}\,\star_{\!\!\! M} \, \gc\,\star_{\!\!\! M} \, \gc^*\,\star_{\!\!\! M} \, \gc^{GD}\,\star_{\!\!\! M} \, \gc\,\star_{\!\!\! M} \, \gc^* \\
   &=& \gc^\dag\,\star_{\!\!\! M} \, \gc\,\star_{\!\!\! M} \, \gc^*\,\star_{\!\!\! M} \, \gc^{GD}\,\star_{\!\!\! M} \, \gc\,\star_{\!\!\! M} \, \gc^*
     =\gc^*\,\star_{\!\!\! M} \, \gc^{GD,*}.
\end{eqnarray*}
In addition, utilization of results in \textbf{(b)} leads to
\begin{align*}
% \nonumber to remove numbering (before each equation)
  \gc^{\dag}\,\star_{\!\!\! M} \, \gc\,\star_{\!\!\! M} \, (\gc^{GD,*})^2\,\star_{\!\!\! M} \, \gc\,\star_{\!\!\! M} \, \gc^{\dag} &=\gc^*\,\star_{\!\!\! M} \, \gc^{GD,*}\,\star_{\!\!\! M} \, \gc\,\star_{\!\!\! M} \, \gc^\dag=\gc^*\,\star_{\!\!\! M} \, \gc^{GD}\,\star_{\!\!\! M} \, \gc\,\star_{\!\!\! M} \, \gc^*\,\star_{\!\!\! M} \, (\gc^*)^\dag\,\star_{\!\!\! M} \, \gc^* \\
     &=\gc^*\,\star_{\!\!\! M} \, \gc^{GD}\,\star_{\!\!\! M} \, \gc\,\star_{\!\!\! M} \, \gc^*=\gc^*\,\star_{\!\!\! M} \, \gc^{GD,*}.
\end{align*}
\textbf{(e)} This is trivial as $\gc\,\star_{\!\!\! M} \, \gc^{GD,*}=\gc\,\star_{\!\!\! M} \, \gc^{GD}\,\star_{\!\!\! M} \, \gc\,\star_{\!\!\! M} \, \gc^*=\gc\,\star_{\!\!\! M} \, \gc^*.$\\
\textbf{(f)} Results of Theorem ${\ref{T15}}$ imply $\gc^{GD,*}\,\star_{\!\!\! M} \, (\gc^\dag)^*=\gc^{GD}\,\star_{\!\!\! M} \, \gc$.
Now, one concludes
\begin{align*}
% \nonumber to remove numbering (before each equation)
  (\gc^\dag)^*\,\star_{\!\!\! M} \, \gc^{GD,*}\,\star_{\!\!\! M} \, (\gc^\dag)^*&=(\gc^\dag)^*\,\star_{\!\!\! M} \, \gc^{GD}\,\star_{\!\!\! M} \, \gc    = (\gc^\dag\,\star_{\!\!\! M} \, \gc\,\star_{\!\!\! M} \, \gc^\dag)^*\,\star_{\!\!\! M} \, \gc^{GD}\,\star_{\!\!\! M} \, \gc \\
   &= (\gc^\dag)^*\,\star_{\!\!\! M} \, \gc^\dag\,\star_{\!\!\! M} \, \gc\,\star_{\!\!\! M} \, \gc^{GD}\,\star_{\!\!\! M} \, \gc
    =(\gc^\dag)^*\,\star_{\!\!\! M} \, \gc^\dag\,\star_{\!\!\! M} \, \gc\\
     &=(\gc^\dag)^*.
\end{align*}
\end{proof}

\begin{theorem}
For given $\gc, \hc\in\mathbb{C}^{\eta_1\times \eta_1\times \eta_3}$ and $\gc\,\star_{\!\!\! M} \, \hc=\hc\,\star_{\!\!\! M} \, \gc$ %$\gc^\dag\,\star_{\!\!\! M} \, \hc^\dag=\hc^\dag\,\star_{\!\!\! M} \, \gc^\dag$.
suppose $k=\max\{\ind(\gc), \ind(\hc)\}$. If
$$\hc\,\star_{\!\!\! M} \, \hc^{GD}\,\star_{\!\!\! M} \, \gc^{GD}=\gc^{GD},\ \ \
 \gc^{GD}\,\star_{\!\!\! M} \, \gc\,\star_{\!\!\! M} \, \hc\,\star_{\!\!\! M} \, \hc^*=\hc\,\star_{\!\!\! M} \, \hc^*\,\star_{\!\!\! M} \, \gc^{GD}\,\star_{\!\!\! M} \, \gc,$$
then $(\gc\,\star_{\!\!\! M} \, \hc)^{GD,*}=\hc^{GD,*}\,\star_{\!\!\! M} \, \gc^{GD,*}$.
\end{theorem}

\begin{proof}
The tensor $\xc=\hc^{GD}\,\star_{\!\!\! M} \, \gc^{GD}$ satisfies
\begin{align*}
 \gc\,\star_{\!\!\! M} \, \hc\,\star_{\!\!\! M} \, \xc\,\star_{\!\!\! M} \, \gc\,\star_{\!\!\! M} \, \hc= \gc\,\star_{\!\!\! M} \, \hc\,\star_{\!\!\! M} \, \hc^{GD}\,\star_{\!\!\! M} \, \gc^{GD}\,\star_{\!\!\! M} \, \gc\,\star_{\!\!\! M} \, \hc
 =\gc\,\star_{\!\!\! M} \, \gc^{GD}\,\star_{\!\!\! M} \, \gc\,\star_{\!\!\! M} \, \hc
 =\gc\,\star_{\!\!\! M} \, \hc
\end{align*}
and
\begin{align*}
 \xc\,\star_{\!\!\! M} \, (\gc\,\star_{\!\!\! M} \, \hc)^{k+1}&=\hc^{GD}\,\star_{\!\!\! M} \, \gc^{GD}\,\star_{\!\!\! M} \, \gc^{k+1}\,\star_{\!\!\! M} \, \hc^{k+1}
 =\hc^{GD}\,\star_{\!\!\! M} \, \gc^k\,\star_{\!\!\! M} \, \hc^{k+1}\\
 &=\hc^{GD}\,\star_{\!\!\! M} \, \hc^{k+1}\,\star_{\!\!\! M} \, \gc^k
 =\hc^k\,\star_{\!\!\! M} \, \gc^k
 =(\gc\,\star_{\!\!\! M} \, \hc)^k.
\end{align*}
Similarly, $(\gc\,\star_{\!\!\! M} \, \hc)^{k+1}\,\star_{\!\!\! M} \, \xc=(\gc\,\star_{\!\!\! M} \, \hc)^k$, Hence, $(\gc\,\star_{\!\!\! M} \, \hc)^{GD}=\hc^{GD}\,\star_{\!\!\! M} \, \gc^{GD}$.
Then,
\begin{align*}
 (\gc\,\star_{\!\!\! M} \, \hc)^{GD,*}&=(\gc\,\star_{\!\!\! M} \, \hc)^{GD}\,\star_{\!\!\! M} \, \gc\,\star_{\!\!\! M} \, \hc\,\star_{\!\!\! M} \, (\gc\,\star_{\!\!\! M} \, \hc)^*\\
 &=\hc^{GD}\,\star_{\!\!\! M} \, \gc^{GD}\,\star_{\!\!\! M} \, \gc\,\star_{\!\!\! M} \, \hc\,\star_{\!\!\! M} \, \hc^*\,\star_{\!\!\! M} \, \gc^*\\
 &=\hc^{GD}\,\star_{\!\!\! M} \, \hc\,\star_{\!\!\! M} \, \hc^*\,\star_{\!\!\! M} \, \gc^{GD}\,\star_{\!\!\! M} \, \gc\,\star_{\!\!\! M} \, \gc^*\\
 &=\hc^{GD,*}\,\star_{\!\!\! M} \, \gc^{GD,*}.
\end{align*}
\end{proof}

\begin{theorem}
Let $\gc, \hc\in\mathbb{C}^{\eta_1\times \eta_1\times \eta_3}$ and $\gc\,\star_{\!\!\! M} \, \hc=\hc\,\star_{\!\!\! M} \, \gc$ and satisfy %$\gc^\dag\,\star_{\!\!\! M} \, \hc^\dag=\hc^\dag\,\star_{\!\!\! M} \, \gc^\dag$.
 $k=\max\{\ind(\gc), \ind(\hc)\}$. If
$$ \gc\,\star_{\!\!\! M} \, \gc^{GD}\,\star_{\!\!\! M} \, \hc^{GD}=\hc^{GD},\ \ \
 \hc^{GD}\,\star_{\!\!\! M} \, \hc\,\star_{\!\!\! M} \, \gc\,\star_{\!\!\! M} \, \gc^*=\gc\,\star_{\!\!\! M} \, \gc^*\,\star_{\!\!\! M} \, \hc^{GD}\,\star_{\!\!\! M} \, \hc,$$
then $(\gc\,\star_{\!\!\! M} \, \hc)^{GD,*}=\gc^{GD,*}\,\star_{\!\!\! M} \, \hc^{GD,*}$.
\end{theorem}

\begin{proof}
Set $\xc=\gc^{GD}\,\star_{\!\!\! M} \, \hc^{GD}$. Now,
\begin{align*}
 \gc\,\star_{\!\!\! M} \, \hc\,\star_{\!\!\! M} \, \xc\,\star_{\!\!\! M} \, \gc\,\star_{\!\!\! M} \, \hc&= \gc\,\star_{\!\!\! M} \, \hc\,\star_{\!\!\! M} \, \gc^{GD}\,\star_{\!\!\! M} \, \hc^{GD}\,\star_{\!\!\! M} \, \gc\,\star_{\!\!\! M} \, \hc
 =\hc\,\star_{\!\!\! M} \, \gc\,\star_{\!\!\! M} \, \gc^{GD}\,\star_{\!\!\! M} \, \hc^{GD}\,\star_{\!\!\! M} \, \gc\,\star_{\!\!\! M} \, \hc\\
 &=\hc\,\star_{\!\!\! M} \, \hc^{GD}\,\star_{\!\!\! M} \, \gc\,\star_{\!\!\! M} \, \hc
 =\hc\,\star_{\!\!\! M} \, \hc^{GD}\,\star_{\!\!\! M} \, \hc\,\star_{\!\!\! M} \, \gc
 =\hc\,\star_{\!\!\! M} \, \gc=\gc\,\star_{\!\!\! M} \,\hc
\end{align*}
and
\begin{align*}
 \xc\,\star_{\!\!\! M} \, (\gc\,\star_{\!\!\! M} \, \hc)^{k+1}&=\gc^{GD}\,\star_{\!\!\! M} \, \hc^{GD}\,\star_{\!\!\! M} \, \gc^{k+1}\,\star_{\!\!\! M} \, \hc^{k+1}
 =\gc^{GD}\,\star_{\!\!\! M} \, \hc^{GD}\,\star_{\!\!\! M} \, \hc^{k+1}\,\star_{\!\!\! M} \, \gc^{k+1}\\
 &=\gc^{GD}\,\star_{\!\!\! M} \, \hc^{k}\,\star_{\!\!\! M} \, \gc^{k+1}
  =\gc^{GD}\,\star_{\!\!\! M} \, \gc^{k+1}\,\star_{\!\!\! M} \, \hc^{k}
 =\gc^k\,\star_{\!\!\! M} \, \hc^k
 =(\gc\,\star_{\!\!\! M} \, \hc)^k.
\end{align*}
Similarly, $(\gc\,\star_{\!\!\! M} \, \hc)^{k+1}\,\star_{\!\!\! M} \, \xc=(\gc\,\star_{\!\!\! M} \, \hc)^k$, Hence, $(\gc\,\star_{\!\!\! M} \, \hc)^{GD}=\gc^{GD}\,\star_{\!\!\! M} \, \hc^{GD}$.
Then,
\begin{align*}
 (\gc\,\star_{\!\!\! M} \, \hc)^{GD,*}&=(\gc\,\star_{\!\!\! M} \, \hc)^{GD}\,\star_{\!\!\! M} \, \gc\,\star_{\!\!\! M} \, \hc\,\star_{\!\!\! M} \, (\gc\,\star_{\!\!\! M} \, \hc)^*
 =\gc^{GD}\,\star_{\!\!\! M} \, \hc^{GD}\,\star_{\!\!\! M} \, \hc\,\star_{\!\!\! M} \, \gc\,\star_{\!\!\! M} \, \hc^*\,\star_{\!\!\! M} \, \gc^*\\
 &=\gc^{GD}\,\star_{\!\!\! M} \, \hc^{GD}\,\star_{\!\!\! M} \, \hc\,\star_{\!\!\! M} \, \gc\,\star_{\!\!\! M} \, \gc^*\,\star_{\!\!\! M} \, \hc^*
 =\gc^{GD}\,\star_{\!\!\! M} \, \gc\,\star_{\!\!\! M} \, \gc^*\,\star_{\!\!\! M} \, \hc^{GD}\,\star_{\!\!\! M} \, \hc\,\star_{\!\!\! M} \, \hc^*
 =\gc^{GD,*}\,\star_{\!\!\! M} \, \hc^{GD,*}.
\end{align*}
\end{proof}

\begin{theorem}
Let $\gc, \hc\in\mathbb{C}^{\eta_1\times \eta_1\times \eta_3}$ and  $\gc\,\star_{\!\!\! M} \, \hc=\hc\,\star_{\!\!\! M} \, \gc=\hc\,\star_{\!\!\! M} \, \gc^*=\mathcal{O}$. Suppose  $k=\max\{\ind(\gc), \ind(\hc)\}$.  If
$$\gc^{GD}\,\star_{\!\!\! M} \, \hc=\hc\,\star_{\!\!\! M} \, \gc^{GD}=\mathcal{O}, \ \ \  \hc^{GD}\,\star_{\!\!\! M} \, \gc=\gc\,\star_{\!\!\! M} \, \hc^{GD}=\mathcal{O},$$
then $(\gc+\hc)^{GD,*}=\gc^{GD,*}+\hc^{GD,*}$.
\end{theorem}

\begin{proof}
Applying the representation of the GD-Star inverse, it can be obtained
\begin{align*}
  (\gc+\hc)^{GD}\,\star_{\!\!\! M} \, (\gc+\hc)\,\star_{\!\!\! M} \, (\gc+\hc)^*&=(\gc^{GD}+\hc^{GD})\,\star_{\!\!\! M} \, (\gc\,\star_{\!\!\! M} \, \gc^*+\gc\,\star_{\!\!\! M} \, \hc^*+\hc\,\star_{\!\!\! M} \, \gc^*+\hc\,\star_{\!\!\! M} \, \hc^*)\\
  &=\gc^{GD}\,\star_{\!\!\! M} \, \gc\,\star_{\!\!\! M} \, \gc^*+\gc^{GD}\,\star_{\!\!\! M} \, \gc\,\star_{\!\!\! M} \, \hc^*\\
  &+\gc^{GD}\,\star_{\!\!\! M} \, \hc\,\star_{\!\!\! M} \, \gc^*+\gc^{GD}\,\star_{\!\!\! M} \, \hc\,\star_{\!\!\! M} \, \hc^*+\hc^{GD}\,\star_{\!\!\! M} \, \gc\,\star_{\!\!\! M} \, \gc^*\\
  &+\hc^{GD}\,\star_{\!\!\! M} \, \gc\,\star_{\!\!\! M} \, \hc^*+\hc^{GD}\,\star_{\!\!\! M} \, \hc\,\star_{\!\!\! M} \, \gc^*+\hc^{GD}\,\star_{\!\!\! M} \, \hc\,\star_{\!\!\! M} \, \hc^*\\
  &=\gc^{GD,*}+\hc^{GD,*}.
\end{align*}
The verification is complete.
\end{proof}

\section{Applications to Multilinear Systems}

In this section we investigate solutions of multilinear systems in terms of the GD inverse, the GDMP inverse and the GD-Star inverse.

\begin{theorem}
Let $\gc\in\mathbb{C}^{\eta_1\times \eta_1\times \eta_3}$ and  $\hc\in\mathbb{C}^{\eta_1\times 1\times \eta_3}$ be given and suppose  $\ind(\gc)=k$.
Then,
\begin{equation}\label{zzgg}
\gc\,\star_{\!\!\! M} \, \xc=\gc\,\star_{\!\!\! M} \, \gc^{GD}\,\star_{\!\!\! M} \, \hc,
\end{equation}
is solvable and $\xc=\gc^{GD}\,\star_{\!\!\! M} \, \hc+(\mathcal{I}-\gc^{GD}\,\star_{\!\!\! M} \, \gc)\,\star_{\!\!\! M} \, \mathcal{Z}$ is solution of (\ref{zzgg}) for arbitrary $\mathcal{Z}\in\mathbb{C}^{\eta_1\times 1\times \eta_3}$.
\begin{proof}
Clearly, $\gc^{GD}\,\star_{\!\!\! M} \, \hc$ is a solution of (\ref{zzgg}).
Hence, the multilinear system (\ref{zzgg}) is consistent. Now,
$$
\gc\,\star_{\!\!\! M} \, \xc =\gc\,\star_{\!\!\! M} \, \gc^{GD} \,\star_{\!\!\! M} \, \hc+\gc\,\star_{\!\!\! M} \, (\mathcal{I}-\gc^{GD}\,\star_{\!\!\! M} \, \gc)\,\star_{\!\!\! M} \, \mathcal{Z}=\gc\,\star_{\!\!\! M} \, \gc^{GD}\,\star_{\!\!\! M} \, \hc.
$$
Moreover, arbitrary solution $\xc$ to (\ref{zzgg}) is generically defined by
$$\xc=\gc^{GD}\,\star_{\!\!\! M} \, \hc+(\mathcal{I}-\gc^{GD}\,\star_{\!\!\! M} \, \gc)\,\star_{\!\!\! M} \, \xc.$$
So,arbitrary solution to (\ref{zzgg}) is determined as
$$\xc=\gc^{GD}\,\star_{\!\!\! M} \, \hc+(\mathcal{I}-\gc^{GD}\,\star_{\!\!\! M} \, \gc)\,\star_{\!\!\! M} \, \mathcal{Z}.$$
\end{proof}

\begin{theorem}
Let $\gc\in\mathbb{C}^{\eta_1\times \eta_1\times \eta_3}$ and  $\hc\in\mathbb{C}^{\eta_1\times 1\times \eta_3}$. Suppose  $\ind(\gc)=k$.  Then, for arbitrary $\mathcal{Z}\in\mathbb{C}^{\eta_1\times 1\times \eta_3}$,
\begin{equation}\label{zzyy}
\gc\,\star_{\!\!\! M} \, \xc=\gc\,\star_{\!\!\! M} \, \gc^\dag\,\star_{\!\!\! M} \, \hc,
\end{equation}
is solvable and $\xc=\gc^{GD,\dag}\,\star_{\!\!\! M} \, \hc+(\mathcal{I}-\gc^{GD}\,\star_{\!\!\! M} \, \gc)\,\star_{\!\!\! M} \, \mathcal{Z}$ is  one of the solutions of (\ref{zzyy}).
\begin{proof}
Clearly, $\gc^{GD}\,\star_{\!\!\! M} \, \gc\,\star_{\!\!\! M} \, \gc^\dag\,\star_{\!\!\! M} \, \hc$ satisfies (\ref{zzyy}). Hence, the system (\ref{zzyy}) is consistent. Now,
$$\gc\,\star_{\!\!\! M} \, \xc
=\gc\,\star_{\!\!\! M} \, \gc^{GD,\dag}
\,\star_{\!\!\! M} \, \hc+\gc\,\star_{\!\!\! M} \, (\mathcal{I}-\gc^{GD}\,\star_{\!\!\! M} \, \gc)\,\star_{\!\!\! M} \, \mathcal{Z}=\gc\,\star_{\!\!\! M} \, \gc^\dag\,\star_{\!\!\! M} \, \hc.
$$
Moreover, if $\xc$ satisies (\ref{zzyy}), then
$$\xc=\gc^{GD,\dag}\,\star_{\!\!\! M} \, \hc+(\mathcal{I}-\gc^{GD}\,\star_{\!\!\! M} \, \gc)\,\star_{\!\!\! M} \, \xc.$$
So, the generic solution to (\ref{zzyy}) is given by
$$\xc=\gc^{GD,\dag}\,\star_{\!\!\! M} \, \hc+(\mathcal{I}-\gc^{GD}\,\star_{\!\!\! M} \, \gc)\,\star_{\!\!\! M} \, \mathcal{Z}.$$
\end{proof}

\end{theorem}

\end{theorem}

\begin{theorem}\label{to}
Let $\gc\in\mathbb{C}^{\eta_1\times \eta_1\times \eta_3}$ and  $\hc\in\mathbb{C}^{\eta_1\times 1\times \eta_3}$ be given and suppose  $\ind(\gc)=k$.
Then,
\begin{equation}\label{zzxx}
\gc\,\star_{\!\!\! M} \, \xc=\gc\,\star_{\!\!\! M} \, \gc^*\,\star_{\!\!\! M} \, \hc,
\end{equation}
is solvable and $\xc=\gc^{GD,*}\,\star_{\!\!\! M} \, \hc+(\mathcal{I}-\gc^{GD}\,\star_{\!\!\! M} \, \gc)\,\star_{\!\!\! M} \, \mathcal{Z}$ satisfies (\ref{zzxx})
for arbitrary $\mathcal{Z}\in\mathbb{C}^{\eta_1\times 1\times \eta_3}$.
\end{theorem}

\begin{proof}
Obviously, $\gc^{GD}\,\star_{\!\!\! M} \, \gc\,\star_{\!\!\! M} \, \gc^*\,\star_{\!\!\! M} \, \hc$ satisfies (\ref{zzxx}).
Hence, the consistency of (\ref{zzxx}) is confirmed.
Now,
$$\gc\,\star_{\!\!\! M} \, \xc
=\gc\,\star_{\!\!\! M} \, \gc^{GD,*} \,\star_{\!\!\! M} \, \hc+\gc\,\star_{\!\!\! M} \, (\mathcal{I}-\gc^{GD} \,\star_{\!\!\! M} \, \gc)\,\star_{\!\!\! M} \, \mathcal{Z}=\gc\,\star_{\!\!\! M} \, \gc^*\,\star_{\!\!\! M} \, \hc.$$
Moreover, if $\xc$ represents a solution to (\ref{zzxx}), then
$$\xc=\gc^{GD,*}\,\star_{\!\!\! M} \, \hc+(\mathcal{I}-\gc^{GD}\,\star_{\!\!\! M} \, \gc)\,\star_{\!\!\! M} \, \xc.$$
In this way, the set of all solutions to (\ref{zzxx}) is specified by
$$\xc=\gc^{GD,*}\,\star_{\!\!\! M} \, \hc+(\mathcal{I}-\gc^{GD}\,\star_{\!\!\! M} \, \gc)\,\star_{\!\!\! M} \, \mathcal{Z}.$$
\end{proof}

\begin{example}
Let $\gc\in\mathbb{C}^{3\times3\times3},~\hc\in\mathbb{C}^{3\times 1\times3},~M\in\mathbb{C}^{3\times3},~\mathcal{Z}\in\mathbb{C}^{3\times1\times3}$ with entries
$$(\mathcal A)^{(1)}=
\begin{bmatrix}
3&0&-1\\
1&-2&0\\
1&0&0
\end{bmatrix},~
(\mathcal A)^{(2)}=
\begin{bmatrix}
2&0&0\\
1&0&1\\
1&0&2
\end{bmatrix},~
(\mathcal A)^{(3)}=
\begin{bmatrix}
1&-2&-3\\
0&-2&-1\\
0&-2&-1
\end{bmatrix},
$$
$$
\mathcal B^{(1)}=
\begin{bmatrix}
4\\
3\\
2
\end{bmatrix},~
\mathcal B^{(2)}=
\begin{bmatrix}
3\\
1\\
0
\end{bmatrix},~
\mathcal B^{(3)}=
\begin{bmatrix}
3\\
0\\
0
\end{bmatrix}, \ \mathcal Z^{(1)}=
\begin{bmatrix}
z_{11}\\
z_{12}\\
z_{13}
\end{bmatrix},~
\mathcal Z^{(2)}=
\begin{bmatrix}
z_{21}\\
z_{22}\\
z_{23}
\end{bmatrix},~
\mathcal Z^{(3)}=
\begin{bmatrix}
z_{31}\\
z_{32}\\
z_{33}
\end{bmatrix}, \
M=
\begin{bmatrix}
1&0&-1\\
0&1&0\\
0&1&-1
\end{bmatrix}.
$$
Firstly, we get the GD and GD-Star inverses by using Algorithm \ref{al} and Algorithm \ref{ao}, that is,
$$
(\gc^{GD})^{(1)}=
\begin{bmatrix}
1.125&1.125&-1.375\\
-0.75&-2.75&3\\
0.125&1.875&-1.125
\end{bmatrix},$$
$$(\gc^{GD})^{(2)}=
\begin{bmatrix}
0.5&0&0\\
0&0&0.25\\
-0.25&0&0.5
\end{bmatrix},~
(\gc^{GD})^{(3)}=
\begin{bmatrix}
1&1&-1.5\\
-1&-2&2.75\\
0.25&1&-1
\end{bmatrix},
$$
$$
(\gc^{GD,*})^{(1)}=
\begin{bmatrix}
2&0&0\\
1.5&-0.75&2.25\\
-2&0&-1
\end{bmatrix},$$
$$(\gc^{GD,*})^{(2)}=
\begin{bmatrix}
2&1&1\\
0.5&0.75&1.25\\
0&1&2
\end{bmatrix},~
(\gc^{GD,*})^{(3)}=
\begin{bmatrix}
-1&-1&-2\\
-0.5&-0.75&0.25\\
-3&-1&-1
\end{bmatrix}.
$$
Then, we calculate $\xc=\gc^{GD,*}
\,\star_{\!\!\! M} \, \hc+(\mathcal{I}-\gc^{GD}\,\star_{\!\!\! M} \, \gc)\,\star_{\!\!\! M} \, \mathcal{Z}$ by Theorem \ref{to}, where
$$
\mathcal X^{(1)}=\begin{bmatrix}
\frac12z_{11}-\frac12z_{12}-\frac12z_{13}+z_{23}+\frac12z_{31}+\frac12z_{32}+\frac32z_{33}+15\\
\frac14z_{21}+z_{22}-z_{23}-\frac12z_{31}+\frac12z_{33}+\frac{27}4\\
-\frac12z_{11}+\frac12z_{12}+\frac12z_{13}+\frac12z_{31}-\frac12z_{32}-\frac12z_{33}+3
\end{bmatrix},$$
$$
\mathcal X^{(2)}=\begin{bmatrix}
7\\
-\frac14z_{21}+z_{22}-\frac12z_{23}+\frac94\\
1
\end{bmatrix},~
\mathcal X^{(3)}=
\begin{bmatrix}
-z_{21}+z_{23}+z_{31}+z_{33}+5\\
\frac14z_{21}+z_{22}-z_{23}-\frac12z_{31}+\frac12z_{33}+\frac34\\
-1
\end{bmatrix}.
$$
\end{example}

\section{Conclusion}

In this work, the GD inverse of tensors under the M-product is presented and investigated.
Firstly, an equivalent expression was proposed applying the core nilpotent decomposition.
The corresponding numerical algorithm for compute this inverse is established.
Subsequently, the reverse order law of the GD inverse is described in detail.
Next, the GDMP inverse of tensors is further explored, providing the corresponding effective numerical algorithms.
The research results demonstrate the properties of the GDMP inverse.
The reverse-order law and forward-order law of the GDMP inverse are also studied.
Finally, the GD-Star inverse under the M-product operation is explained, and a numerical calculation method for the GD-Star inverse is established, along with relevant properties.
Based on the above content, solutions for the GD inverse, the GDMP inverse, and the GD-star inverse of multilinear equations are also provided, and an example program is developed to calculate the corresponding GD-star inverse solution.

\medskip
\noindent {\bf Funding.} Hongwei Jin is supported by the Special Fund for Science and Technological Bases and Talents of Guangxi (No. GUIKE AD21220024).
Predrag Stanimirovi\'c is supported by the Ministry of Science, Technological~Development and Innovation, Republic of Serbia, Contract No. 451-03-65/2024-03/200124 and
and by the Science Fund~of~the~Repub\-lic of Serbia, Grant No.~7750185, Quantitative Automata Models: Fundamental Problems and Applications -- QUAM.

\smallskip
\noindent {\bf Conflict of interest}~~~No potential conflict of interest was reported by the authors.

\smallskip
\noindent {\bf Acknowledgements.} The authors would like to thank the editor and the anonymous reviewers for their valuable comments, which
have significantly improved the paper.

%The authors would like to thank the editor and the anonymous reviewers for their valuable comments, which
%have significantly improved the paper.

%\end{spacing}


\begin{thebibliography}{99}


\bibitem{K1} E. Kilmer, C. Martin. Factorization strategies for third-order tensors. Linear Algebra and its Applications, 435 (2011), pp. 641-658.

\bibitem{J1}
X. Jin. Developments and applications of block Toeplitz iterative solvers. Springer Science  Business Media, 2003.

\bibitem{KK}
E. Kernfeld, M. Kilmer, S. Aeron,
Tensor-tensor products with invertible linear transforms, Linear Algebra Appl, 485 (2015), pp. 545-570.

\bibitem{A1}
E. Albert, W. Perrett, G. Jeffery. The foundation of the general theory of relativity. Annalen der Physik, 354 (1916), 769.

\bibitem{Shao1}
J. Shao. A general product of tensors with applications. Linear Algebra and its applications, 439 (2013), pp. 2350-2366.

\bibitem{L1}
L. Qi, Z. Luo. Tensor analysis: spectral theory and special tensors. Society for Industrial and Applied Mathematics, 2017.




\bibitem{KBHH}
M. Kilmer, K. Braman, N. Hao, R. Hoover, Third-order tensors as operators on matrices: a theoretical and computational framework with applications
in imaging, SIAM J. Matrix Anal. Appl. 34 (2013), pp. 148-172.


\bibitem{Chen} J. Chen, W. Ma, Y. Miao, Y. Wei,
Perturbations of Tensor-Schur decomposition and its applications to multilinear control systems and facial recognitions, Neurocomputing, 547 (2023), 126359.

\bibitem{Che} M. Che, X. Wang, Y. Wei, X. Zhao,
Fast randomized tensor singular value thresholding for low-rank tensor optimization, Numer. Linear Algebra Appl, 29 (2022), e2444.


\bibitem{QW}
Q. Wang, X. Xu, Iterative algorithms for solving some tensor equations. Linear and Multilinear Algebra, 67 ( 2019), pp. 1325-1349.

\bibitem{kong1}
H. Kong, X. Xie, Z. Lin. t-Schatten-$ p $ norm for low-rank tensor recovery. IEEE Journal of Selected Topics in Signal Processing, 12 (2018), pp. 1405-1419.

\bibitem{L3}
Y. Liu, L. Chen, C. Zhu. Improved robust tensor principal component analysis via low-rank core matrix. IEEE Journal of Selected Topics in Signal Processing, 12 (2018), pp. 1378-1389.

\bibitem{Hu}
 W. Hu, Y. Yang, W. Zhang, et al. Moving object detection using tensor-based low-rank and saliently fused-sparse decomposition. IEEE Transactions on Image Processing, 26 (2016), pp. 724-737.

\bibitem{Long}
Z. Long, Y. Liu, L. Chen, et al. Low rank tensor completion for multiway visual data. Signal processing, 155 (2019), pp. 301-316.

\bibitem{A2}
A. Wang, Z. Lai, Z. Jin, Noisy low-tubal-rank tensor completion, Neurocomputing 330 (2019), pp. 267-279.


\bibitem{MSL}
C. Martin, R. Shafer, B. LaRue,
An Order-$p$ Tensor Factorization with Applications in Imaging, SIAM J. Scientific Computing, 35 (2013),
pp. 474-490.

\bibitem{Soltani}
S. Soltani, M. Kilmer, P. Hansen, A tensor-based dictionary learning approach to tomographic image reconstruction, BIT Numer. Math. 56 (2016), pp. 1425-1454.

\bibitem{Tarzanagh}
D. Tarzanagh, G. Michailidis, Fast randomized algorithms for t-product based tensor operations and decompositions with applications to imaging data, SIAM J. Imaging Sci. 11 (2018), pp. 2629-2664.

\bibitem{1}
 T. Chan, T. Yang, Polar n-complex and n-bicomplex singular value decomposition and principal component pursuit, IEEE Trans. Signal Process. 64 (2016), pp. 6533-6544.

 \bibitem{RU} L. Reichel, O. Ugwu,
Tensor Krylov subspace methods with an invertible linear transform product applied to image processing, Appl. Numer. Math., 166 (2021), pp. 186-207.


\bibitem{2}
T. Liu, L. Chen, C. Zhu, Improved robust tensor principal component analysis via low-rank core matrix, IEEE J. Sel. Top. Signal Process. 12 (2018), pp. 1378-1389.

\bibitem{3}
 Z. Long, Y. Liu, L. Chen, C. Zhu, Low rank tensor completion for multiway visual data, Signal Process. 155 (2019), pp. 301-316.

\bibitem{4}
 N. Hao, M. Kilmer, K. Braman, R. Hoover, Facial recognition using tensor-tensor decompositions, SIAM J. Imaging Sci. 6 (2013), pp. 437-463.
 \bibitem{5}
 W. Hu, Y. Yang, W. Zhang, Y. Xie, Moving object detection using tensor-based low-rank and saliently fused-sparse decomposition, IEEE Trans. Image Process. 26 (2017), pp. 724-737.

\bibitem{Wang4} X. Wang, C. Mo, S. Qiao, Y. Wei,
Predefined-time convergent neural networks for solving the time-varying nonsingular multi-linear tensor equations, Neurocomputing, 472 (2022), pp. 68-84.


\bibitem{Wang2} X. Wang, M. Che, C. Mo, Y. Wei,
Solving the system of nonsingular tensor equations via randomized Kaczmarz-like method, Journal of Computational and Applied Mathematics, 421 (2023), pp. 114-856.

\bibitem{Wei} P. Wei, X. Wang, Y. Wei,
Neural network models for time-varying tensor complementarity problems, Neurocomputing, 523 (2023), pp. 18-32.


%\bibitem{KK}
%E. Kernfeld, M. Kilmer, S. Aeron,
%Tensor-tensor products with invertible linear transforms, Linear Algebra Appl., 485 (2015), pp. 545-570.

\bibitem{K3} M. Kilmer, L. Horesh, H. Avron, E. Newman,
Tensor-tensor products for optimal representation and compression, arXiv:2001.00046, 2020.






\bibitem{KB} T. Kolda, B.  Bader,
Tensor decompositions and applications, SIAM review, 51 (2009), pp. 455-500.



\bibitem{KM}  M. Kilmer, C. Martin, Factorization strategies for third-order tensors, Linear Algebra Appl., 434 (2011), pp. 641-658.

\bibitem{Wang} X. Wang, M. Che, Y. Wei,
Tensor neural network models for tensor singular value decompositions, Comput. Optim. Appl., 75 (3) (2020), pp. 753-777.

\bibitem{Miao3} Y. Miao, T. Wang, Y. Wei,
Stochastic conditioning of tensor functions based on the tensor-tensor product, Pacific Journal of Optimization, 19 (2) (2023), pp. 205-235.

%\bibitem{Mo1} C. Mo, X. Wang, Y. Wei,
%Time-varying generalized tensor eigenanalysis via Zhang neural networks, Neurocomputing, 407 (2020), pp. 465-479.


\bibitem{Wang1} X. Wang, M. Che, Y. Wei,
Randomized Kaczmarz methods for tensor complementarity problems, Computational Optimization and Applications, 82 (2022), pp. 595-615.


\bibitem{Wang3} X. Wang, P. Wei, Y. Wei,
A Fixed Point Iterative Method for Third-order Tensor Linear Complementarity Problems, Journal of Optimization Theory and Applications, 197 (1) (2023), pp. 334-357.

\bibitem{Shao} X. Shao, Y. Wei, J. Yuan,
Nonsymmetric Algebraic Riccati Equations under the Tensor Product, Numerical Functional Analysis and Optimization, 44 (6) (2023), pp. 545-563.


\bibitem{Sun}
L. Sun, B. Zheng, C. Bu, Y. Wei, Moore-Penrose inverse of tensors via Einstein product, Linear
Multilinear Algebra, 64 (4) (2016), pp. 686-698.

\bibitem{Miao}
Y. Miao, L. Qi, Y. Wei, Generalized tensor function via the tensor singular value decomposition based
on the T-product, Linear Algebra Appl., 590 (2020), pp. 258-303.

\bibitem{Miao2} Y. Miao, L. Qi, Y. Wei, T-Jordan Canonical Form and T-Drazin Inverse based on the T-Product, Com. Appl. Math. Comput., 3 (2021), pp. 201-220.


%\bibitem{Ben-Israel} A. Ben-Israel, T. Greville,
%{Generalized Inverse: Theory and Applications}, 2nd ed., Springer, New York, 2003.

\bibitem{Ji2}
J. Ji, Y. Wei, The Drazin inverse of an even-order tensor and its application to singular tensor equations,
Comput. Math. Appl., 75 (9) (2018), pp. 3402-3413.

\bibitem{Jin}
H. Jin, M. Bai, J. Ben\'itez, X. Liu, The generalized inverses of tensors and an application to linear models, Comput. Math. Appl., 74 (2017), pp. 385-397.



\bibitem{Liang}
M. Liang, B. Zheng, Further results on Moore-Penrose inverses of tensors with application to tensor
nearness problems, Comput. Math. Appl., 77 (5) (2019), pp. 1282-1293.


\bibitem{Che1} M. Che, Y. Wei,
An efficient algorithm for computing the approximate t-URV and its applications, Journal of Scientific Computing, 92 (3) (2022), pp. 557-583.

\bibitem{Cong} Z. Cong, H. Ma,
Acute perturbation for Moore-Penrose inverses of tensors via the T-product, J. Appl. Math. Comput., 68 (6) (2022), pp. 3799-3820.




\bibitem{Behera}
R. Behera, D. Mishra, Further results on generalized inverses of tensors via the Einstein product, Linear Multilinear Algebra, 65 (8) (2017), pp. 1662-1682.




\bibitem{Liu} Y. Liu, H. Ma,
Dual core generalized inverse of third-order dual tensor based on the T-product, Comput. Appl. Math., 41 (8) (2022), pp. 391(1-28).


\bibitem{Cong2} Z. Cong, H. Ma,
Characterizations and perturbations of the core-EP inverse of tensors based on the T-product, Numer. Funct. Anal. Optim., 43 (10) (2022), pp. 1150-1200.

\bibitem{Sahoo} J. Sahoo, R. Behera, P. S. Stanimirovi\'c, V. Katsikis, H. Ma,
Core and core-EP inverses of tensors, Comput. Appl. Math., 39 (1) (2020), pp. 1-28.

\bibitem{Cao} Z. Cao, P. Xie,
Perturbation analysis for t-product-based tensor inverse, Moore-Penrose inverse and tensor system, Communications on Applied Mathematics and Computation, 4 (4) (2022), pp. 1441-1456.

\bibitem{Cao1} Z. Cao, P. Xie,
On some tensor inequalities based on the t-product, Linear and Multilinear Algebra, 71 (3) (2023), pp. 377-390.


\bibitem{Liu1} Y. Liu, H. Ma,
Weighted generalized tensor functions based on the tensor-product and their applications, Filomat, 36 (18) (2022), pp. 6403-6426.



%\bibitem{Mo2} C. Mo, W. Ding, Y. Wei, Perturbation analysis on T-eigenvalues of third-order tensors,
%arXiv:2108.09502v2, 2023.

\bibitem{SPBS} J. Sahoo, S. Panda, R. Behera, P. S. Stanimirovi\'{c}, Computation of tensors generalized inverses under M-product and applications, 	arXiv:2405.16111, 2024.



\bibitem{JXWL} H. Jin, S. Xu, Y. Wang, X. Liu. The Moore-Penrose inverse of tensors via the M-Product. Computational and Applied Mathematics. 42 (6) (2023), 294.

\bibitem{[386]} J. K. Sahoo, S. K. Panda, R. Behera, P.S. Stanimirovi\'c, Computation of tensors generalized inverses under M-product and applications, Journal of Mathematical Analysis and Applications, 542 (2025), doi: 10.1016/j.jmaa.2024.128864.

%\bibitem{Krushnachandra}
%P. Krushnachandra, B. Ratikanta, M. Debasisha, Reverse order law for the Moore-Penrose inverses of tensors,
%Linear Multilinear Algebra, 68 (2020), pp. 246-264.





\bibitem{WL} H. Wang, X. Liu. Partial orders based on core-nilpotent decomposition, Linear algebra and its applications, 488 (2016),  pp. 235-248.


\end{thebibliography}
\end{document}